\documentstyle[11pt,bezier]{article}
\input{psfig}

\begin{document}

\font\bbbld=msbm10 scaled\magstep1
\newcommand{\bfR}{\hbox{\bbbld R}}
\newcommand{\bfC}{\hbox{\bbbld C}}
\newcommand{\bfZ}{\hbox{\bbbld Z}}
\newcommand{\bfH}{\hbox{\bbbld H}}
\newcommand{\bfQ}{\hbox{\bbbld Q}}
\newcommand{\bfN}{\hbox{\bbbld N}}
\newcommand{\bfP}{\hbox{\bbbld P}}
\newcommand{\bfT}{\hbox{\bbbld T}}
\def\Sym{\mathop{\rm Sym}}
\newcommand{\halo}[1]{\Int(#1)}
\def\Int{\mathop{\rm Int}}
\def\Re{\mathop{\rm Re}}
\def\Im{\mathop{\rm Im}}
\newcommand{\union}{\cup}
\newcommand{\goesto}{\rightarrow}
\newcommand{\bdy}{\partial}
\newcommand{\n}{\noindent}
\newcommand{\p}{\hspace*{\parindent}}

\newtheorem{theorem}{Theorem}[section]
\newtheorem{assertion}{Assertion}[section]
\newtheorem{proposition}{Proposition}[section]
\newtheorem{lemma}{Lemma}[section]
\newtheorem{definition}{Definition}[section]
\newtheorem{claim}{Claim}[section]
\newtheorem{corollary}{Corollary}[section]
\newtheorem{observation}{Observation}[section]
\newtheorem{conjecture}{Conjecture}[section]
\newtheorem{question}{Question}[section]
\newtheorem{example}{Example}[section]

\newbox\qedbox
\setbox\qedbox=\hbox{$\Box$}
\newenvironment{proof}{\smallskip\noindent{\bf Proof.}\hskip \labelsep}%
                        {\hfill\penalty10000\copy\qedbox\par\medskip}
\newenvironment{remark}{\smallskip\noindent{\bf Remark.}\hskip \labelsep}%
                        {\hfill\penalty10000\copy\qedbox\par\medskip}
\newenvironment{remark1}{\smallskip\noindent{\bf Remark 1.}\hskip \labelsep}%
                        {\hfill\penalty10000\copy\qedbox\par\medskip}
\newenvironment{remark2}{\smallskip\noindent{\bf Remark 2.}\hskip \labelsep}%
                        {\hfill\penalty10000\copy\qedbox\par\medskip}
\newenvironment{proofspec}
         {\smallskip\noindent{\bf Proof of Theorem 6.1.}\hskip \labelsep}%
                        {\hfill\penalty10000\copy\qedbox\par\medskip}
\newenvironment{acknowledgements}{\smallskip\noindent{\bf Acknowledgements.}%
        \hskip\labelsep}{}

\setlength{\baselineskip}{1.0\baselineskip}

\title{The Morse Index of Wente Tori}
\author{{\sc Wayne Rossman}\\
{Faculty of Science, Kobe Univ., Kobe, Japan 657-8501}} 
\date{}
\maketitle

\begin{abstract}
We find various lower and upper bounds for the index of Wente tori 
that contain a continuous family of planar principal curves.  
We then prove a result that gives an algorithm for computing 
the index sharply.  
\end{abstract}

\section{Introduction}

Hopf conjectured that every constant mean curvature (CMC) closed 
(closed = compact, without boundary) 
surface in $\bfR^3$ is a round sphere.  This is 
true if the surface is assumed to be 
either genus 0 or embedded \cite{H}, but in general 
it is false, and the first counterexamples (with genus 1) 
were found by Wente \cite{We}.  
Abresch \cite{A} then gave a more explicit representation for 
those Wente tori which contain a continuous family of planar principal 
curves, referred to here as {\em symmetric} Wente tori.  
Walter \cite{Wa} later found an even more explicit 
representation for symmetric Wente tori, by noticing that if one 
family of 
principal curves are all planar, then the perpendicular 
principal curves each lie in a sphere.  Finally, Spruck \cite{Sp} showed 
that these symmetric Wente tori are exactly the same surfaces 
that Wente originally found.  
Pinkall and Sterling \cite{PS} and Bobenko \cite{B} went on to 
classify all closed 
CMC tori in $\bfR^3$, and Kapouleas constructed closed 
CMC surfaces for every genus greater than one \cite{K1}, \cite{K2}.

The index of minimal surfaces in $\bfR^3$ has been well studied.  
In this case there are no closed surfaces.  
Do Carmo and Peng \cite{CP} showed that the only stable (index 0) 
complete minimal surface is a plane.  Fischer-Colbrie 
\cite{FC} showed that minimal surfaces have finite index if and only 
if they have finite total curvature.  And for many minimal surfaces 
with finite total curvature the index has been found 
or has known bounds (for example, see \cite{FC}, \cite{N}, 
\cite{MR}, \cite{T}, \cite{Cho}).  

However, less is known about the index of nonminimal CMC surfaces.  
It is known that 
the surface is stable if and only if it is a round sphere \cite{BC} 
(analogous to the do Carmo-Peng result).  
It is also known that CMC surfaces without boundary 
in $\bfR^3$ have finite index if and only if they are 
compact \cite{LR}, \cite{Si} (analogous to Fischer-Colbrie's result).  
However, other than the round sphere, there is no 
closed CMC surface with known index.  
(Unlike CMC surfaces in general, minimal surfaces have meromorphic 
Gauss maps when given conformal coordinates.  
Furthermore, for the index of nonminimal CMC surfaces we 
must consider a volume 
constraint that does not appear in the minimal case.  
This accounts for why more progress has been made in finding 
the index of minimal surfaces.)  

Our purpose here is to expand the collection of nonminimal 
CMC surfaces for which
we have specific estimates of the index (to include surfaces other 
than just round spheres).  We 
restrict ourselves to symmetric Wente tori, because they  
are the simplest known examples of nonspherical 
closed CMC surfaces in $\bfR^3$, and can be represented nicely 
\cite{A}, \cite{Wa}.  Hence they are the obvious 
candidates to consider in an initial investigation.  

In section 2, we briefly reiterate Walter's description of 
symmetric Wente tori, describe the Jacobi operator and the 
index problem, and discuss the 
eigenvalues and eigenfunctions of the Laplacian on a torus (which 
play an essential role in our discussion).  

In section 3, we use Courant's nodal domain theorem to find 
initial lower bounds for the index, obtaining as a corollary that 
for every natural number $N$, there exist only finitely many 
symmetric Wente tori with index less than $N$.  
In section 4, we use upper and lower bounds of the Jacobi operator's 
potential function to determine upper and lower bounds for the index.  

In section 5, we use eigenfunctions of the 
Laplacian to create function spaces on which the Jacobi operator 
is negative definite, producing much sharper lower bounds 
for the index.  These bounds imply that 
every symmetric Wente torus has index at least 7.  
Furthermore, 
the two simplest Wente tori (the first two surfaces shown in 
Figure 1) have index at least $9$ and $8$, respectively.  

In section 6, we prove the correctness of 
an algorithm (a variation of the finite element method \cite{FS}) 
for computing the index 
exactly, up to a possible change by 1 from the volume constraint.  
In this proof, we essentially show that 
on a flat torus, the spectrum of a Schrodinger operator
$\triangle + V$ with $C^\infty$ potential function $V$ can be 
determined from the restrictions of the operator 
to function spaces 
spanned by a finite number of explicitly known eigenfunctions of 
the Laplacian $\triangle$.  

Finally, in section 7, we implement the algorithm 
to give numerical estimates for the index of $17$ different 
Wente tori (Table 3).  These numerical estimates show 
that the lower bounds of section 5 are quite sharp, and show that it is 
reasonable to conjecture that 
every symmetric Wente torus has index at least $9$.  

\begin{acknowledgements}
The author thanks S. Nayatani, H. Urakawa, K. Akutagawa, and M. 
Koiso for their helpful suggestions.  
\end{acknowledgements}

\section{Preliminaries}

\subsection{Walter's representation [Wa]} 

Consider a smooth conformal immersion
${\cal X}: {\cal M} \rightarrow \bfR^3$, where 
${\cal M}$ is a compact 
2-dimensional torus with the induced metric.  If 
$(x,y)$ are isothermal coordinates on 
$\cal M$, then the
first and second fundamental forms,
Gaussian curvature, mean curvature, 
and Hopf differential are 
\[ ds^2 = E(dx^2 + dy^2)\; , \; \; \; 
    II = L dx^2 + 2M dx dy + N dy^2 \; , \]\[
  K = \frac{LN-M^2}{E^2} \; , \; \; \; 
   H = \frac{L+N}{2 E} \; , \; \; \; 
   \Phi = \frac{1}{2} (L-N) - iM \; . \] 
The Gauss and Codazzi equations imply that 
$\Phi$ is holomorphic with respect
to $z=x+iy$ when $H$ is constant.  

Since $\cal M$ is a torus, we may 
assume that ${\cal M} = \bfC / \Gamma$, where 
$\Gamma$ is a lattice in the plane $\bfC$.  Since 
$dz$ is a global 1-form on $\bfC / \Gamma$, and since 
we will assume $H$ is constant, we 
see that $\Phi dz^2$ is a constant multiple of $dz^2$.  

We define a function $F$ by \[ H E = e^F \; . \]
By a linear transformation of $\bfC$, we may arrange 
that $\Phi dz^2 = dz^2$, hence 
$M = 0$, $L = e^F+1$, and $N = e^F-1$.  
Thus $\cal X$ has no umbilic points, and 
$(x,y)$ become curvature line coordinates.  

Denoting the oriented unit normal by $\vec{N}$, we 
may compute that 
\[ {\cal X}_{xx} = \frac{1}{2} F_x {\cal X}_x - 
\frac{1}{2} F_y {\cal X}_y - (e^F + 1) \vec{N} \; , \; \; \; 
{\cal X}_{xy} = \frac{1}{2} F_y {\cal X}_x + 
\frac{1}{2} F_x {\cal X}_y \; , \]
\[ {\cal X}_{yy} = - \frac{1}{2} F_x {\cal X}_x + 
\frac{1}{2} F_y {\cal X}_y - (e^F - 1) \vec{N} \; , \; \; \; 
\vec{N}_{x} = H (1+e^{-F}) {\cal X}_x \; , \; \; \; 
\vec{N}_{y} = H (1-e^{-F}) {\cal X}_y \; , \]
\begin{equation}
\triangle F + 4 H \sinh{F} = 0 \; , 
\end{equation}
where $\triangle = \frac{\partial^2}{\partial x^2} + 
\frac{\partial^2}{\partial y^2}$.  Thus finding the 
immersion ${\cal X}: {\bfC / \Gamma} \rightarrow \bfR^3$ is 
reduced to solving for $F$ in equation (2.1).  
In the case of symmetric Wente tori, 
Walter determined $F$ and the immersion $\cal X$ explicitly.  
In the process of doing so, he proved the following lemmas.  

\begin{lemma}
The set of all symmetric Wente tori are in a one-to-one 
correspondence with the set of reduced 
fractions $\ell/n \in (1,2)$.  
\end{lemma}

\begin{figure}
\centerline{
        \hbox{
		\psfig{figure=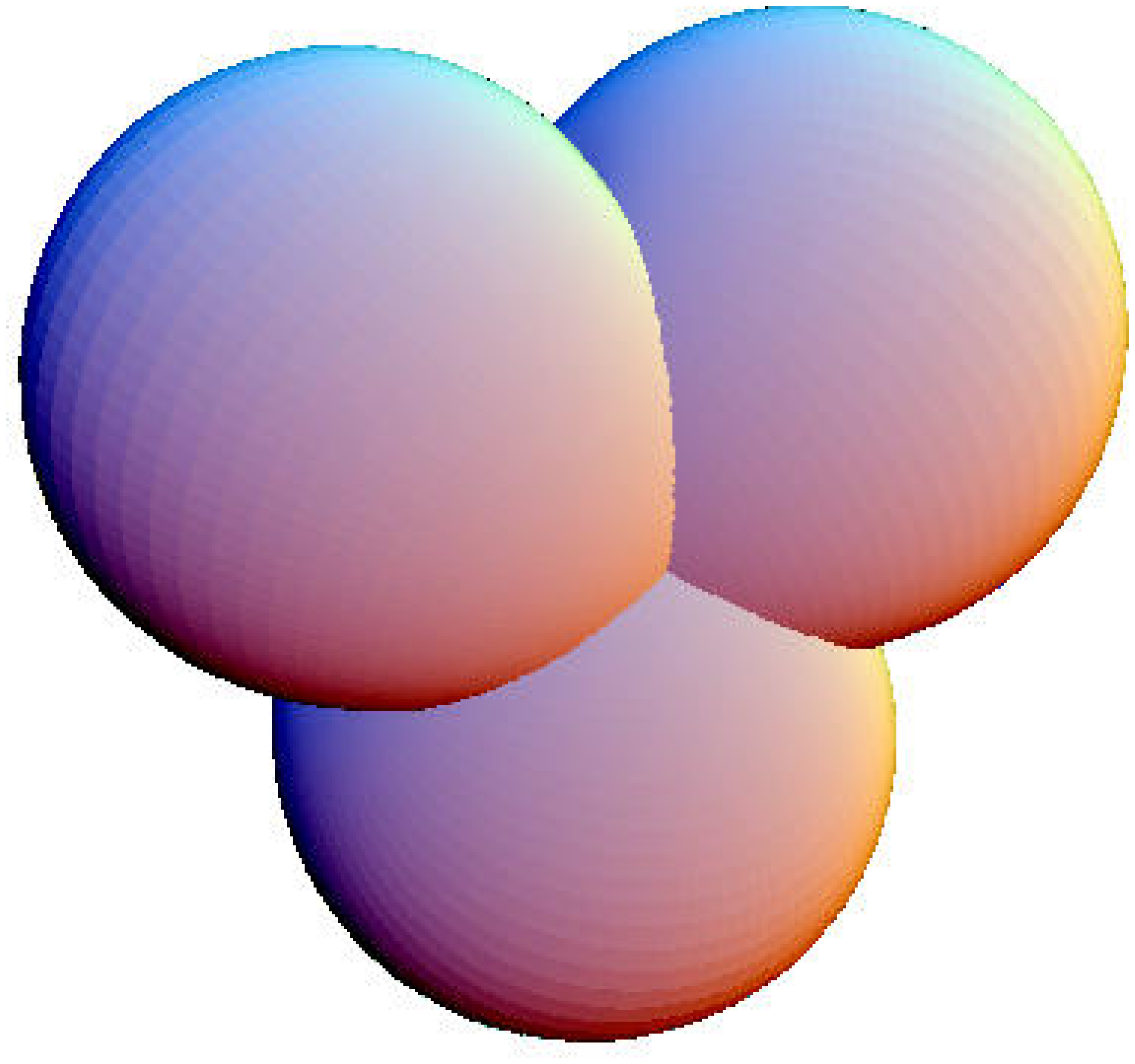,width=1.7in}
		\hspace{0.2cm}
		\psfig{figure=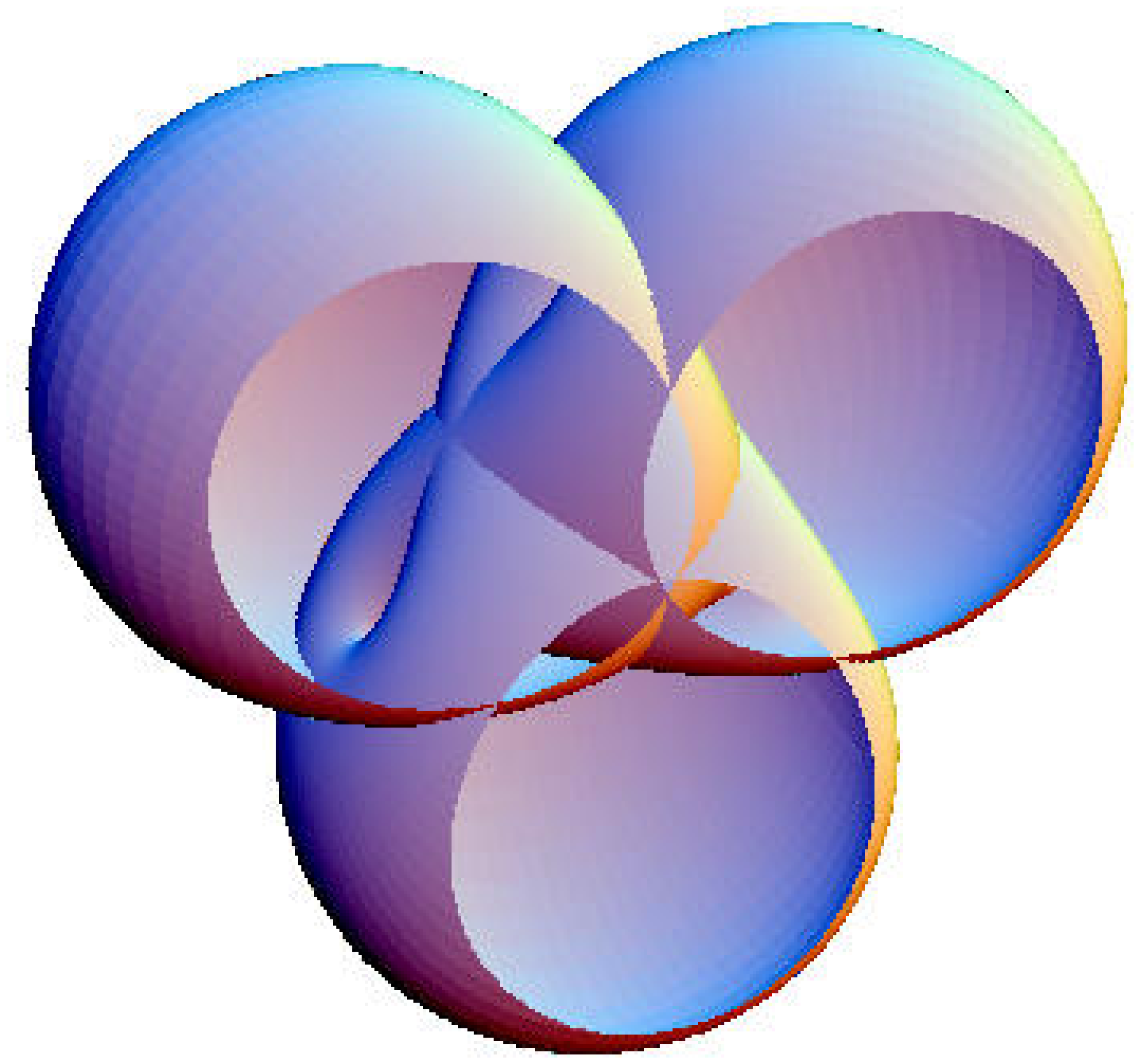,width=1.7in}
		\hspace{0.2cm}
		\psfig{figure=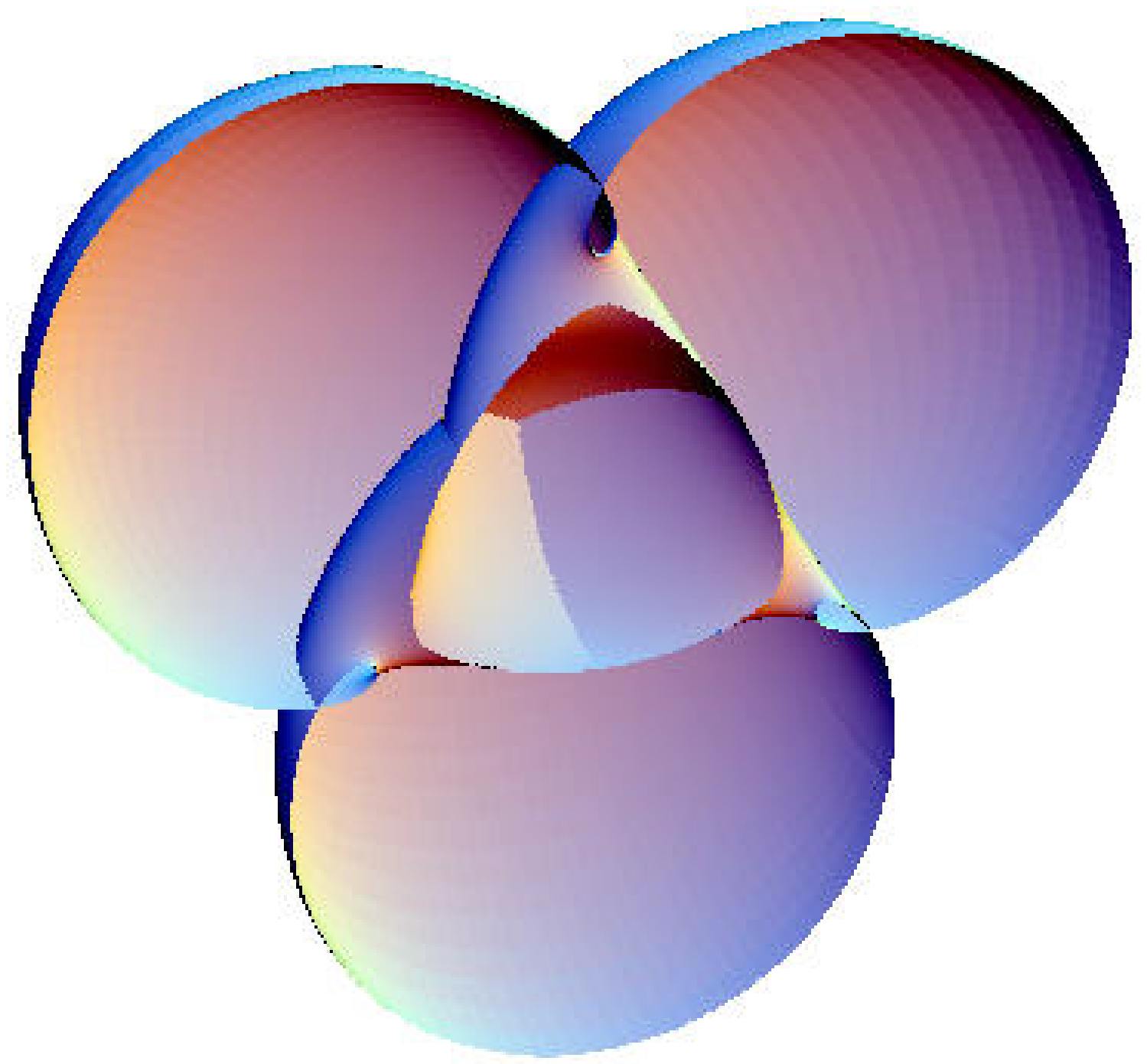,width=1.7in}
	}}
\centerline{
        \hbox{
		\psfig{figure=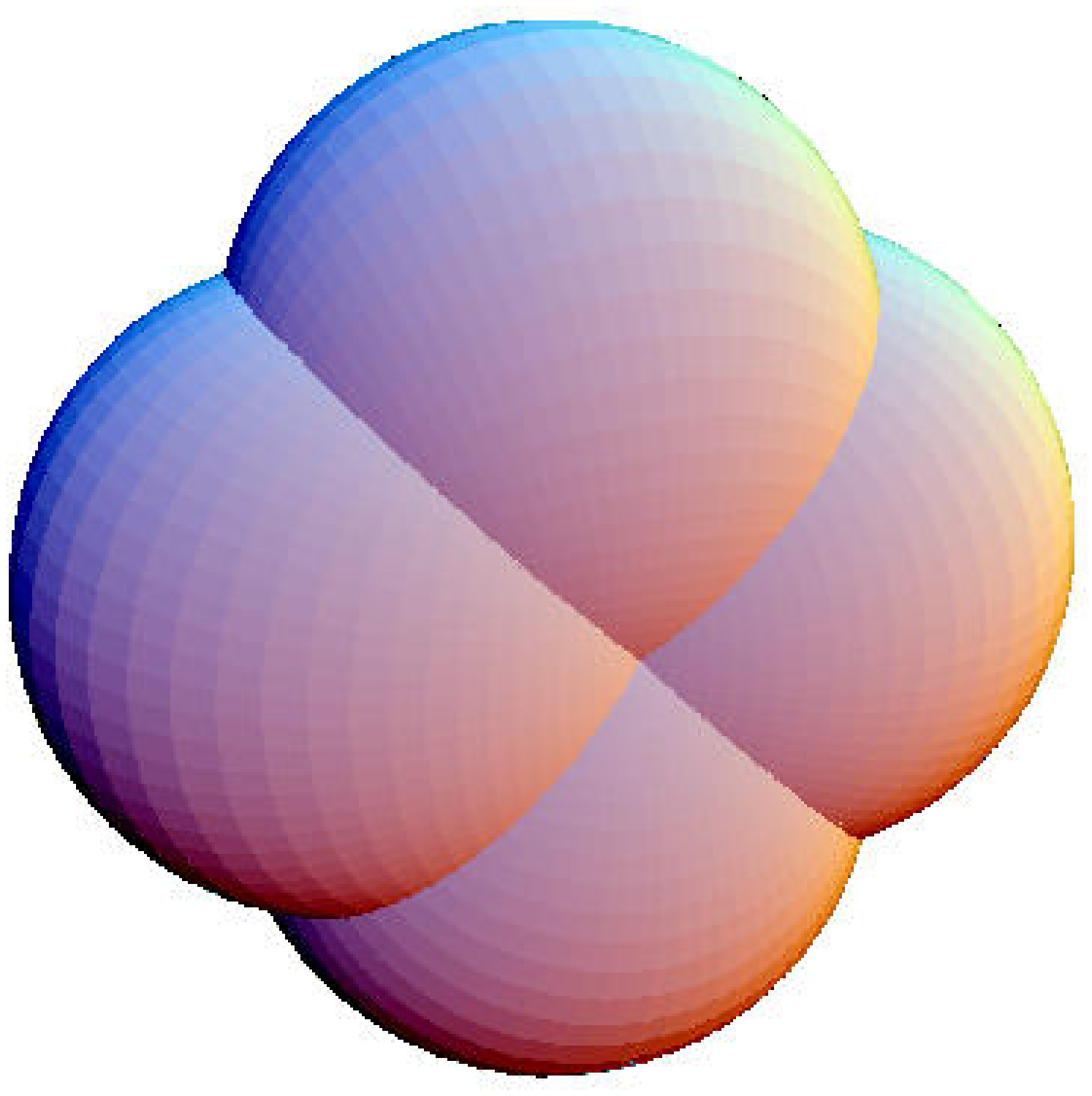,width=1.7in}
		\hspace{0.2cm}
		\psfig{figure=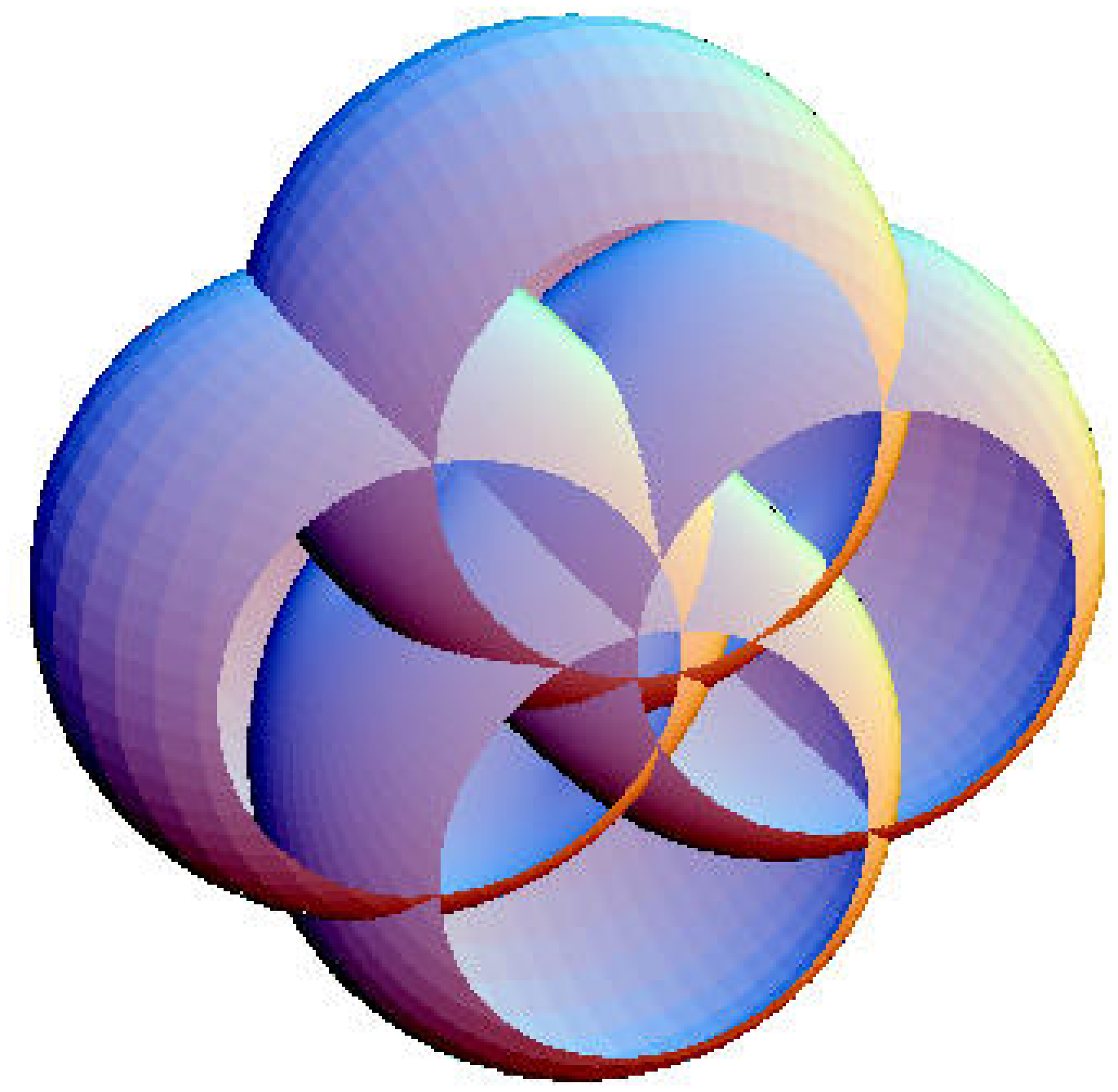,width=1.7in}
		\hspace{0.2cm}
		\psfig{figure=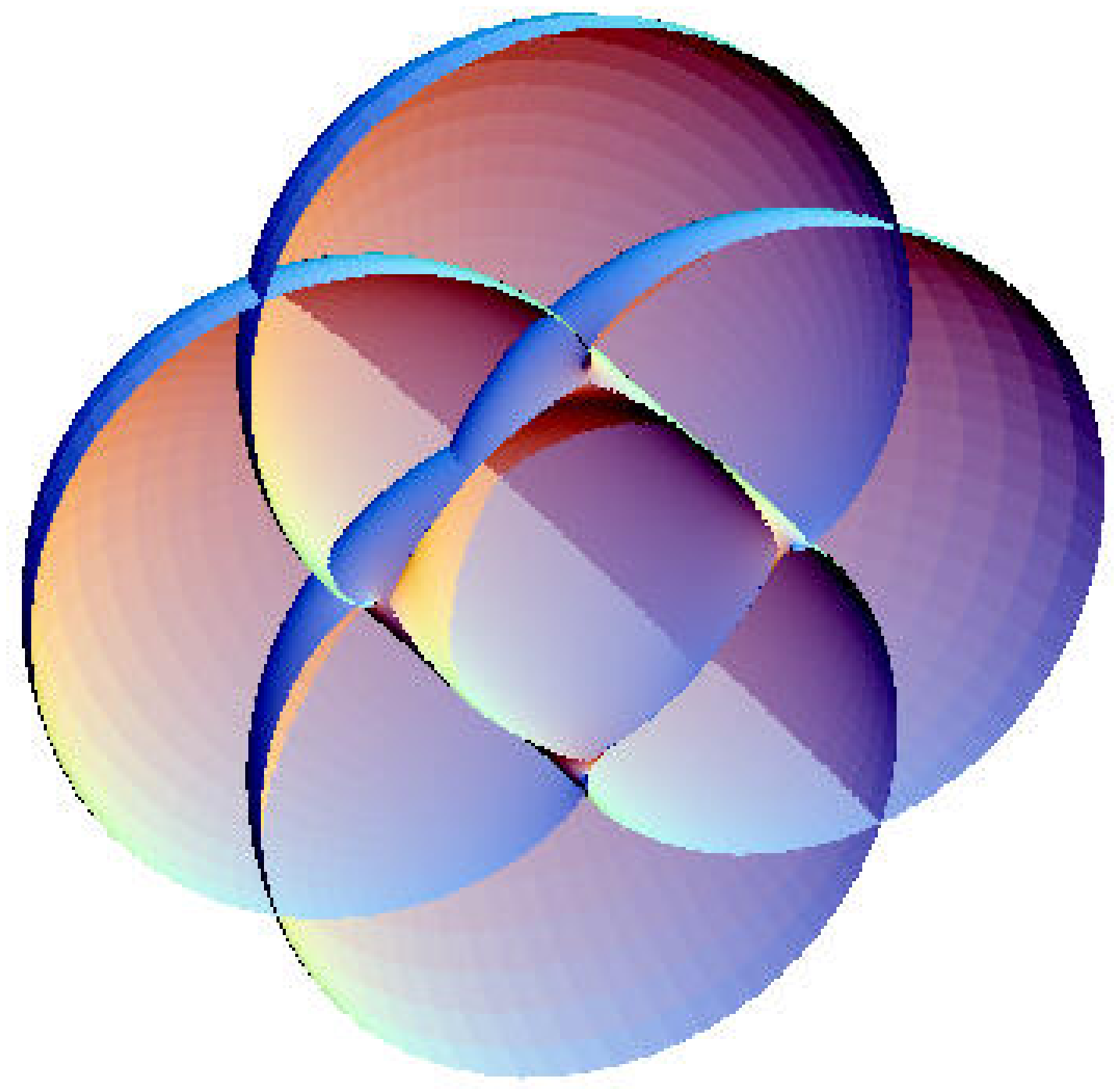,width=1.7in}
	}}
\centerline{
        \hbox{
		\psfig{figure=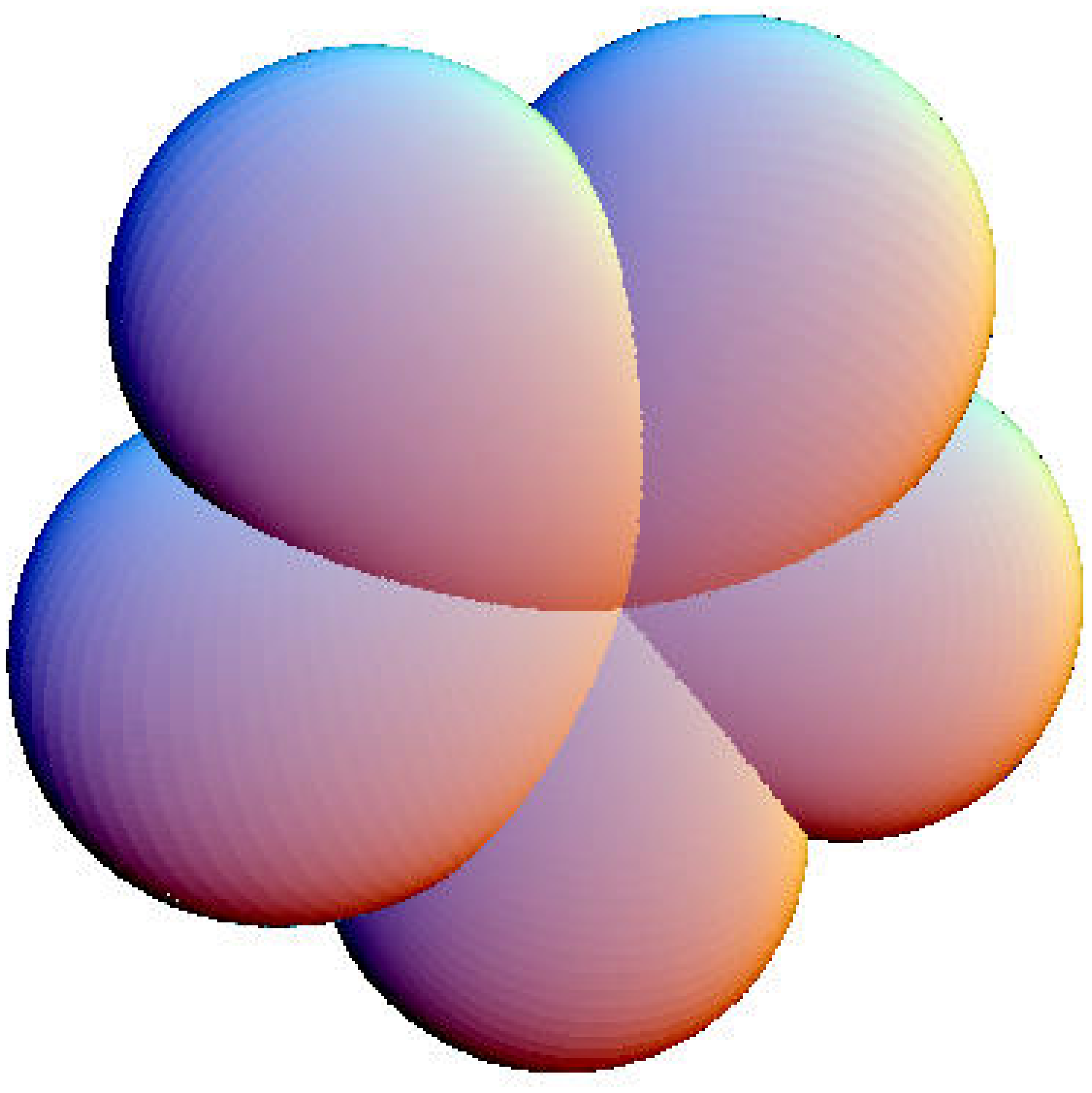,width=1.7in}
		\hspace{0.2cm}
		\psfig{figure=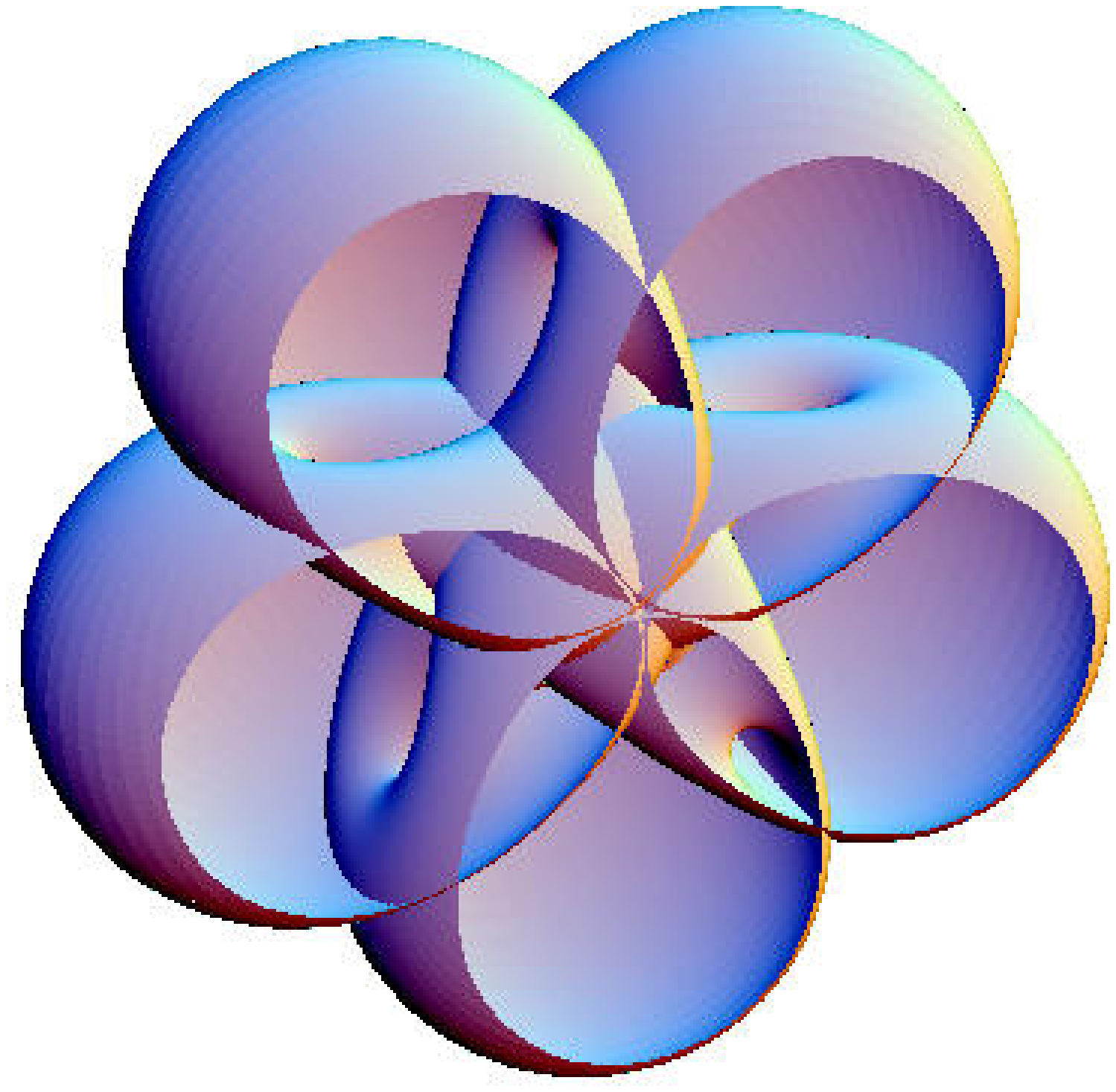,width=1.7in}
		\hspace{0.2cm}
		\psfig{figure=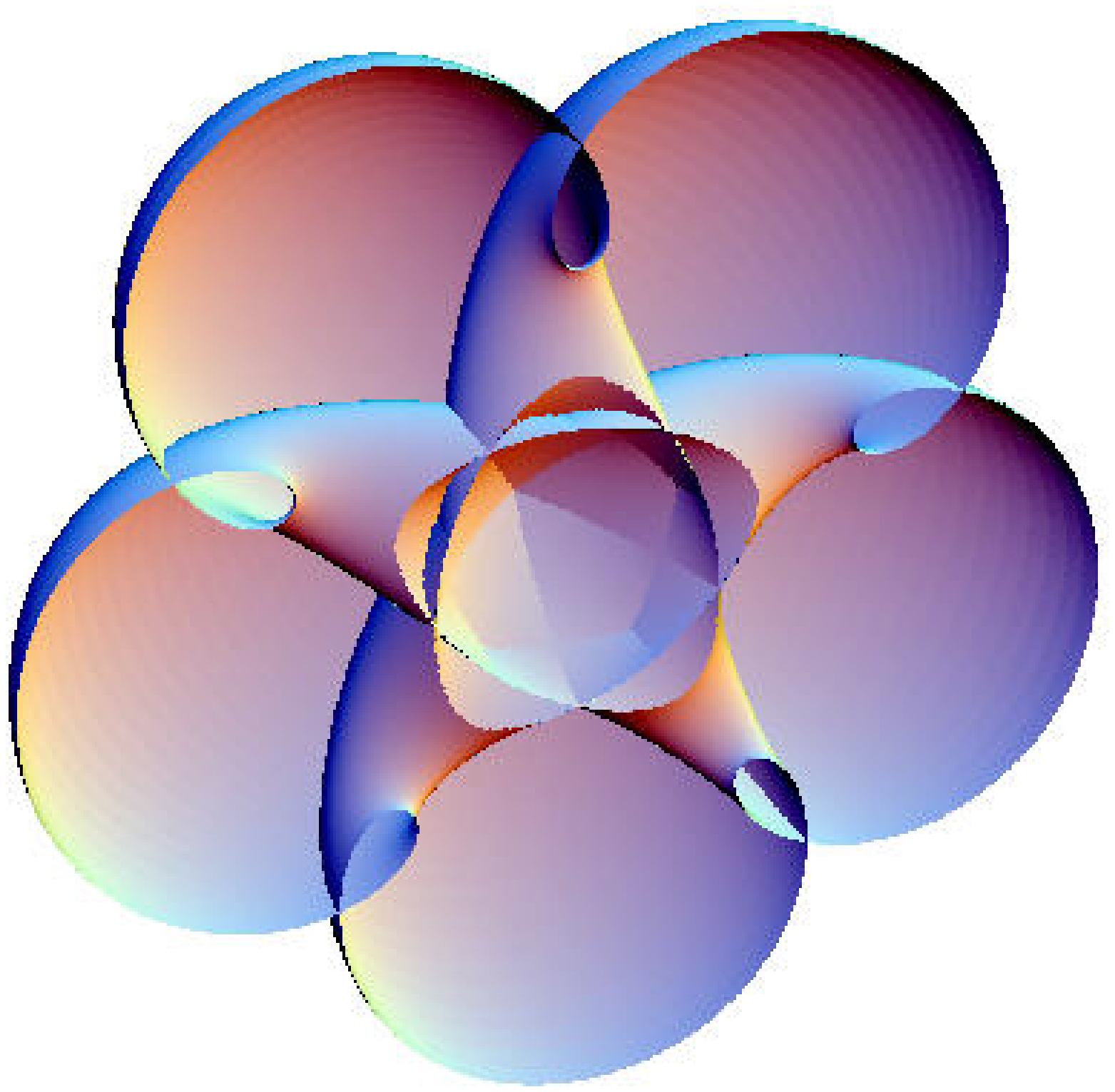,width=1.7in}
	}}
\caption{The surfaces ${\cal W}_{4/3}$, ${\cal W}_{3/2}$, 
and ${\cal W}_{6/5}$, with cut-aways to expose the inner parts.  
(Graphics by Katsunori Sato of Tokyo 
Institute of Technology.)}
\end{figure}

For each $\ell/n$, we call the corresponding symmetric 
Wente torus ${\cal W}_{\ell/n}$.  Each ${\cal W}_{\ell/n}$ has 
either one or two planar geodesic loops in the 
central symmetry plane: two loops 
if $\ell$ is odd, and one loop if $\ell$ is even.  Each 
loop can be partitioned into $2 n$ congruent curve segments, 
and $\ell$ is the total winding order of the Gauss map along each 
loop.  (The central planar geodesic loops are contained in the 
boundaries of the right-hand side figures in Figure 1.)  

\begin{lemma}
The function $F$ for the surface 
${\cal W}_{\ell/n}$ is 
\[ F = 4 \mbox{arctanh}(f(x) g(y)) \; , \mbox{ with } \]
\[ f(x) = \gamma \mbox{cn}_k(\alpha x) \; , \; \; \; 
g(y) = \bar{\gamma} \mbox{cn}_{\bar{k}}(\bar{\alpha} y) 
\; , \]
where $\mbox{cn}_k$ (resp. $\mbox{cn}_{\bar{k}}$) is 
the amplitudinus cosinus of 
Jacobi with modulus $k$ (resp. $\bar{k}$), and 
\[ k = \sin{\theta}, \bar{k} = \sin{\bar{\theta}} \; , 
\mbox{ for } \theta,\bar{\theta} \in (0,\pi/2) \; , 
\mbox{ and } \theta + \bar{\theta} < \frac{\pi}{2} \; , 
\]\[ \gamma = \sqrt{\tan{\theta}} > 0 \; , \; \; \; 
\bar{\gamma} = \sqrt{\tan{\bar{\theta}}} > 0 \; , \; \; \; 
\alpha = \sqrt{4 H \frac{\sin{2 \bar{\theta}}}
{\sin{2(\theta+\bar{\theta})}}} \; , \; \; \; 
\bar{\alpha} = \sqrt{4 H \frac{\sin{2 \theta}}
{\sin{2(\theta+\bar{\theta})}}} \; . \] 
\end{lemma}

There are two period problems for the immersion.  One is a 
translational period problem in the direction of the 
rotation axis of the surface, and the other is a rotational 
period problem around the rotation axis.  
Using elliptic function theory, 
Walter determined that for all $\ell/n$, to solve the 
translational period problem of the surface, one needs 
$\bar{\theta} = 65.354955354^\circ$.  To then solve the 
rotational period problem, there is a unique 
choice of $\theta \in (0^\circ,24.645044646^\circ)$.  
The value of $\theta$ (but not $\bar{\theta}$) 
will depend on $\ell/n$ (see Table 2).  

Let $x_{\ell n}$ (resp. $y_{\ell n}$) be the length of the 
period of $f(x)$ (resp. $g(y)$).  

\begin{lemma}
${\cal X}: \bfC/\Gamma \to {\cal W}_{\ell/n}$ is a conformal 
diffeomorphism , where 
\[ \Gamma = 
\mbox{span}_{\bfZ} \{(n x_{\ell n},0),(0,y_{\ell n})\}
\; \; \; \; \mbox{ when } 
\ell \mbox{ is odd, and} 
\] 
\[ \Gamma = 
\mbox{span}_{\bfZ} \{(n x_{\ell n}/2,y_{\ell n}/2),
(0,y_{\ell n})\} 
\; \; \; \; \mbox{ when } 
\ell \mbox{ is even.} 
\] 
The curves $\{[x_0,y] \; | \, x_0 = constant \}$ are 
mapped by $\cal X$ to planar curvature lines of 
${\cal W}_{\ell/n}$.  
\end{lemma}

The reason why $\Gamma$ is different for 
$\ell$ odd or even is best explained in the 
introduction of \cite{A}.  (However, the 
notation "$m/n$" in \cite{A} is only half the value 
of the notation "$\ell/n$" used here and in \cite{Wa}.  
One can easily check that the characterization of 
$\Gamma$ depends on the denominator's parity in 
\cite{A}'s notation, and depends on the 
numerator's parity in \cite{Wa}'s notation.)  

$\cal X$ maps the domain 
$\{ (x,y) \; | \; 0 \leq x \leq x_{\ell n}/2, 0 \leq y 
\leq y_{\ell n}/2 \}$ to a fundamental 
piece of ${\cal W}_{\ell/n}$, 
bounded by four planar geodesics.  The entire 
surface is composed of pieces 
congruent to the fundamental piece, in the sense 
that if one continues the fundamental piece by 
reflecting across planes containing planar 
geodesics, one arrives at the entire ${\cal W}_{\ell/n}$.  
When $\ell$ is odd (resp. $\ell$ is even), 
${\cal W}_{\ell/n}$ consists of the 
union of $4 n$ (resp. $2 n$) fundamental pieces.  

\subsection{Index and the associated eigenvalue problem}

Let $\vec{v}_t$ be the variation vector field 
at time $t$ of a $C^\infty$ variation $M(t) \subset 
\bfR^3$ of the surface ${\cal W}_{\ell/n}$ so that 
$M(0) = {\cal W}_{\ell/n}$. Let 
$A(t)$ be the induced area of $M(t)$, and let 
\[ u_t := \langle \vec{v}_t, \vec{N} \rangle_{\bfR^3} \; , 
\; \; \; u := u_0 \; . \]  
It is natural to consider only those variations for 
which $\int_{{\bfC}/\Gamma} u_t dA_t = 0$ ($dA_t$ is the 
induced area form on $M(t)$, $dA := dA_0$) for all $t$ close 
to $0$, corresponding 
to the fact that the volume inside a "soap bubble" is preserved by 
any physically natural variation of the bubble.  Such a variation 
is {\em volume preserving} \cite{BC}.  So, the first variation 
formula implies that for volume preserving variations, 
\[ A^{\prime} (0) = H \int_{{\bfC}/\Gamma} u dA = 0 \; . \]

The second variation formula for volume preserving 
variations (\cite{BC}) is 
\[ A^{\prime \prime} (0) = 
\int_{{\bfC}/\Gamma} |\nabla u|^2 + (2K-4H^2) u^2 dA = 
\int_{{\bfC}/\Gamma} u ((-\triangle + 2KE-4H^2E)(u)) dxdy 
\; , \mbox{ so} \]  
\begin{equation} A^{\prime \prime} (0) = 
\int_{\bfC / \Gamma} u {\cal L}(u) dxdy \; , \; \; \; 
\mbox{ where} \end{equation}
\[ {\cal L}(u) = - \triangle u - V \hspace{-0.03in} \cdot u 
\hspace{-0.045in} \; , \; \; \; V := 4H \cosh{F} \]
is the Jacobi operator with potential function $V$.  
(This operator is also computed in \cite{PS}, but with 
notation that differs by constant factors.)  

Note that in equation (2.2), we 
are integrating with respect to the flat metric on $\bfC / 
\Gamma$, not with respect to the metric induced by 
the immersion.  

\begin{definition}
The index Ind(${\cal W}_{\ell/n}$) of ${\cal W}_{\ell/n}$ is the 
maximum possible dimension 
of a subspace ${\cal V} \subset C^\infty(\bfC/\Gamma)$ such 
that $\int_{\bfC / \Gamma} u dA = 0$ and 
$\int_{\bfC / \Gamma} u{\cal L}u dxdy < 0$ for all nonzero 
$u \in {\cal V}$.  
\end{definition}

By Lemma 2.4 of \cite{BC}, for every $u$ satisfying 
$\int_{\bfC / \Gamma} u dA = 0$, there exists a volume preserving 
variation of ${\cal W}_{\ell/n}$ with variation vector field 
$u \vec{N}$ on ${\cal W}_{\ell/n}$.  Thus, loosely speaking, 
Ind(${\cal W}_{\ell/n}$) is the maximum dimension of a space of 
variation vector fields for 
area-decreasing volume-preserving smooth variations, that is, a 
space on which $A^\prime(0) = 0$ and 
$A^{\prime \prime}(0) < 0$ for all its nonzero 
variation vector fields.  

Let $L^2 = L^2(\bfC/\Gamma)$ denote the measurable functions $u$ on 
$\bfC / \Gamma$ (or equivalently on ${\cal W}_{\ell/n}$) satisfying 
$\int_{\bfC / \Gamma} u^2 dxdy < 
\infty$, and define the $L^2$ inner product $\langle u,v 
\rangle_{L^2} := \int_{\bfC / \Gamma} uv dxdy$.  
If a sequence of functions $u_i \in L^2$ converges strongly 
in the $L^2$ norm to a function $u$, we denote this 
by $u_i \to_{L^2} u$ as $i \to \infty$.  Note that we 
define this $L^2$ space with respect to the flat 
metric on $\bfC / \Gamma$, not the metric induced by the 
immersion.  The following theorem is well known (see, for example, 
\cite{U}):  

\begin{theorem}
The operator ${\cal L}=-\triangle - V$ on $\bfC / \Gamma$ has 
a discrete spectrum of eigenvalues 
\[ \beta_1 \leq \beta_2 \leq .... \nearrow +\infty \] (each considered 
with multiplicity 1), and has corresponding 
eigenfunctions \[v_1,v_2,... \; \in C^\infty(\bfC/\Gamma) \] 
which form an orthonormal basis for $L^2$.  
Furthermore, we have the following variational characterization for 
the eigenvalues:
\begin{equation} \beta_j = \inf_{{\cal V}_j} \left( \sup_{\phi \in 
{\cal V}_j, || \phi 
||_{L^2} = 1} \int_{\bfC / \Gamma} \phi {\cal L} \phi dxdy \right) \; , 
\end{equation}
where ${\cal V}_j$ runs through all $j$ dimensional subspaces of 
$C^\infty (\bfC/\Gamma)$.  
\end{theorem}

\begin{lemma}
If ${\cal L}$ has k negative eigenvalues, then 
Ind(${\cal W}_{\ell/n}$) is either $k$ or $k-1$.
\end{lemma}

\begin{proof} 
Since $\beta_{k+1} \geq 0$, from characterization (2.3) with $j=k+1$, 
we see there is no $k+1$ dimensional 
subspace of $C^\infty(\bfC/\Gamma)$ on which 
$\int_{\bfC / \Gamma} u{\cal L}u dxdy < 0$ for all nonzero 
functions $u$ in the $k+1$ dimensional subspace.  Thus 
Ind(${\cal W}_{\ell/n}$)$\leq k$.  

Now let ${\cal V}_k=\mbox{span}\{v_1,...,v_k\}$.  Since 
$\beta_j < 0$ for all $j \leq k$, 
${\cal V}_k$ is a $k$ dimensional space satisfying 
$\int_{\bfC/\Gamma} u{\cal L}u dxdy < 0$ for all nonzero $u \in 
{\cal V}_k$.  Define a linear map ${\cal F}: {\cal V}_k \to \bfR$ by 
${\cal F}(u) = \int_{\bfC / \Gamma} u dA$.  The kernel Ker($\cal F$)$= 
\{u \in {\cal V}_k \, | \, \int_{\bfC / \Gamma} u dA = 0 \}$ is a 
subspace of ${\cal V}_k$ of dimension at least $k-1$.  Then, by 
definition, Ind(${\cal W}_{\ell/n}$)$\geq$dim(Ker($\cal F$))$\geq k-1$.  
\end{proof}

\begin{remark}
It also follows from Lemma 2.4's proof that 
if there exists a $j$ dimensional space ${\cal V}_j \subset 
C^\infty(\bfC/\Gamma)$ with 
$\int_{\bfC/\Gamma} u {\cal L}u dxdy < 0$ for all nonzero 
$u \in {\cal V}_j$, then Ind(${\cal W}_{\ell/n}$)$\geq j-1$.  
\end{remark}

We define the {\em nullity} Null(${\cal L}$) of 
${\cal L}$ to be the multiplicity 
of the eigenvalue $0$ of ${\cal L}$.  

\begin{lemma}
Null(${\cal L}$)$\geq 6$ for every ${\cal W}_{\ell/n}$.  
\end{lemma}

\begin{proof} 
For any vector field $\vec{v}$ on $\bfR^3$ generated by a 
rigid motion of $\bfR^3$, 
the normal part $\langle \vec{v},\vec{N} 
\rangle_{\bfR^3} \vec{N}$ on ${\cal W}_{\ell/n}$ is a Jacobi 
field on ${\cal W}_{\ell/n}$, that is, 
${\cal L}(\langle \vec{v},\vec{N} \rangle_{\bfR^3}) = 0$.  
Suppose Null(${\cal L}$)$\leq 5$.  Then there is 
some such $\vec{v}$ entirely tangent to 
${\cal W}_{\ell/n}$, since the rigid motions of $\bfR^3$ form 
a 6 dimensional group.  Choose such a $\vec{v}$.  For each
$p \in {\cal W}_{\ell/n}$, let $\phi_p(t)$ be the 
integral curve 
of $\vec{v}$ in ${\cal W}_{\ell/n}$ such that $\phi_p(0) = p$.  
Since $d\phi_p(t)/dt = \vec{v}_{\phi_p(t)}$, 
$\phi_p(t)$ is also an integral curve of $\vec{v}$ in $\bfR^3$, 
hence ${\cal W}_{\ell/n}$ is invariant under the rigid 
motion that generates $\vec{v}$. But  
${\cal W}_{\ell/n}$ is not invariant under any 
rigid motion, a contradiction, proving the lemma.  
\end{proof}

\subsection{Eigenvalues of the Laplacian}

By Lemma 2.4, our goal becomes to compute the number of 
negative eigenvalues of ${\cal L}$.  In order to do this, 
we use a convenient fact: we know explicitly the complete set 
of eigenvalues $\alpha_i$ and eigenfunctions $u_i$ of $- \triangle$.  
For $\bfC/\Gamma$, with 
$\Gamma=\mbox{span}_{\bfZ} \{(a_1,a_2),(b_1,b_2)\}$, the 
complete set of eigenvalues of $- \triangle u_i = 
\alpha_i u_i$ are 
\[ \frac{4\pi^2}{(a_1b_2 - a_2b_1)^2} \left( (m_2b_2-m_1a_2)^2+
(m_1a_1-m_2b_1)^2 \right) \; \; , \] with corresponding 
eigenfunctions \[ c_{m_1,m_2} \cdot 
(\sin \, \mbox{{\scriptsize or}} \, \cos) \left(
\frac{2\pi}{a_1b_2 - a_2b_1} ((m_2b_2-m_1a_2)x+
(m_1a_1-m_2b_1)y) \right) \; \; , \]
for $m_1,m_2 \in \bfZ$.  We want these 
eigenfunctions to have $L^2$-norm equal to 1, so we choose 
$c_{m_1,m_2} = \sqrt{2/(|a_1b_2-a_2b_1|)}$ if $|m_1|+|m_2|>0$ and 
$c_{0,0}=\sqrt{1/(|a_1b_2-a_2b_1|)}$.  

When $\ell$ is odd, by Lemma 2.3, we have
$a_1=nx_{\ell n}$, $a_2=b_1=0$, $b_2=y_{\ell n}$, and 
we can list these eigenvalues and 
eigenfunctions by choosing $m_1$ and 
$m_2$ in the following order: $(m_2,m_1)=(0,0)$, 
$(1,0)$, $(0,1)$, $(2,0)$, $(1,1)$, 
$(1,-1)$, $(0,2)$, $(3,0)$, $(2,1)$, $(2,-1)$, $(1,2)$, 
$(1,-2)$, $(0,3)$, $... \; $, always choosing sine 
first and cosine second.  
We can continue in this way to define all $\alpha_i$ and 
$u_i$.  The eigenfunctions $u_i$ form an orthonormal 
basis for $L^2$.  

When $\ell$ is even, by Lemma 2.3, we have 
$a_1=nx_{\ell n}/2$, $a_2= y_{\ell n}/2$, 
$b_1=0$, $b_2=y_{\ell n}$, and again we can list the eigenvalues and 
eigenfunctions by choosing $m_1$ and 
$m_2$ so that $(2m_2-m_1,m_1)$ has the following order: 
$(2m_2-m_1,m_1)=(0,0)$, $(2,0)$, $(1,1)$, 
$(1,-1)$, $(0,2)$, $(4,0)$, $(3,1)$, $(3,-1)$, $(2,2)$, 
$(2,-2)$, $(1,3)$, $(1,-3)$, $(0,4)$, $... \; $, always choosing 
sine first and cosine second.  
Again the eigenfunctions $u_i$ will form an orthonormal 
basis for $L^2$.  

Note that, with these orderings, we do not 
necessarily have $\alpha_i \leq 
\alpha_{i^\prime}$ for $i<i^\prime$.  However, we still have 
$\lim_{i \to \infty} \alpha_i = + \infty$.  

Since we later refer to the $\alpha_i$ and $u_i$ without restating 
their definitions, an unambiguous understanding of their 
ordering is needed, so for clarity we list 
the first $13$ $\alpha_i$, $u_i$ in Table 1.  

\begin{table}[htb]
    \begin{center}
    \begin{tabular}{c|c|c|c}
 eigenvalues  & eigenfunctions & eigenvalues  & 
                eigenfunctions\\
 $\alpha_i$ for $\ell$ odd & $u_i$ for 
$\ell$ odd & $\alpha_i$ for $\ell$ even &
               $u_i$ for $\ell$ even \\
\hline
 $\alpha_1 = 0 $  & $u_1 = \frac{1}
{\sqrt{n x_{\ell n} y_{\ell n}}}$ & 
$\alpha_1 = 0 $  & $u_1 = \frac{\sqrt{2}}
{\sqrt{n x_{\ell n} y_{\ell n}}}$  \\ \hline
$\alpha_2 = \frac{4 \pi^2}{n^2x_{\ell n}^2} $ & 
$u_2 = 
\frac{\sin(2 \pi x/(nx_{\ell n}))}
{\sqrt{n x_{\ell n} y_{\ell n}/2}} $ & 
$\alpha_2 = \frac{16 \pi^2}{n^2x_{\ell n}^2} $ & 
$u_2 = 
\frac{\sin(4 \pi x/(nx_{\ell n}))}
{\sqrt{n x_{\ell n} y_{\ell n}/4}} $  \\ \hline
$\alpha_3 = \frac{4 \pi^2}{n^2x_{\ell n}^2} $ & 
$u_3 = 
\frac{\cos(2 \pi x/(nx_{\ell n}))}{\sqrt{n 
x_{\ell n} y_{\ell n}/2}} $ & 
$\alpha_3 = \frac{16 \pi^2}{n^2x_{\ell n}^2} $ & 
$u_3 = 
\frac{\cos(4 \pi x/(nx_{\ell n}))}{\sqrt{n 
x_{\ell n} y_{\ell n}/4}} $  \\ \hline
$\alpha_4 
= \frac{4 \pi^2}{y_{\ell n}^2} $ & $u_4 = 
\frac{\sin(2 \pi y/y_{\ell n})}{\sqrt{n x_{\ell n} 
y_{\ell n}/2}} $ &
$\alpha_4 
= \frac{4 \pi^2}{n^2x_{\ell n}^2}+ 
\frac{4 \pi^2}{y_{\ell n}^2} $ & $u_4 = 
\frac{\sin(
\frac{2 \pi x}{nx_{\ell n}}+
\frac{2 \pi y}{y_{\ell n}})}{\sqrt{n x_{\ell n} 
y_{\ell n}/4}} $  \\ \hline
$\alpha_5 
= \frac{4 \pi^2}{y_{\ell n}^2} $ & $u_5 = 
\frac{\cos(2 \pi y/y_{\ell n})}{\sqrt{n 
x_{\ell n} y_{\ell n}/2}} $ & 
$\alpha_5 
= \frac{4 \pi^2}{n^2x_{\ell n}^2}+ 
\frac{4 \pi^2}{y_{\ell n}^2} $ & $u_5 = 
\frac{\cos(
\frac{2 \pi x}{nx_{\ell n}}+
\frac{2 \pi y}{y_{\ell n}})}{\sqrt{n x_{\ell n} 
y_{\ell n}/4}} $  \\ \hline
$\alpha_6 = \frac{16 \pi^2}{n^2x_{\ell n}^2} $ & 
$u_6 = 
\frac{\sin(4 \pi x/(nx_{\ell n}))}{\sqrt{n 
x_{\ell n} y_{\ell n}/2}} $ & 
$\alpha_6 = \frac{4 \pi^2}{n^2x_{\ell n}^2}+ 
\frac{4 \pi^2}{y_{\ell n}^2} $ & 
$u_6 = 
\frac{\sin(\frac{2 \pi x}{nx_{\ell n}}-
\frac{2 \pi y}{y_{\ell n}})}{\sqrt{n 
x_{\ell n} y_{\ell n}/4}} $  \\ \hline
$\alpha_7 = \frac{16 \pi^2}{n^2x_{\ell n}^2} $ & 
$u_7 = 
\frac{\cos(4 \pi x/(nx_{\ell n}))}{\sqrt{n 
x_{\ell n} y_{\ell n}/2}} $ & 
$\alpha_7 = \frac{4 \pi^2}{n^2x_{\ell n}^2}+ 
\frac{4 \pi^2}{y_{\ell n}^2} $ & 
$u_7 = 
\frac{\cos(\frac{2 \pi x}{nx_{\ell n}}-
\frac{2 \pi y}{y_{\ell n}})}{\sqrt{n 
x_{\ell n} y_{\ell n}/4}} $  \\ \hline
$\alpha_8 = 
\frac{4 \pi^2 }{n^2x_{\ell n}^2}+\frac{4 \pi^2 }{y_{\ell n}^2} $ & 
$u_8 = 
\frac{\sin(  \frac{2 \pi x}{nx_{\ell n}}
 +\frac{2 \pi y}{y_{\ell n}} )}
{\sqrt{n x_{\ell n} y_{\ell n}/2}} $ & 
$\alpha_8 = 
\frac{16 \pi^2 }{y_{\ell n}^2} $ & 
$u_8 = 
\frac{\sin(4 \pi y/y_{\ell n})}
{\sqrt{n x_{\ell n} y_{\ell n}/4}} $  \\ \hline
$\alpha_9 = 
\frac{4 \pi^2 }{n^2x_{\ell n}^2}+\frac{4 \pi^2 }{y_{\ell n}^2} $ & 
$u_9 = 
\frac{\cos(\frac{2 \pi x}{nx_{\ell n}}+\frac{2 \pi y}{y_{\ell n}})}
{\sqrt{n x_{\ell n} y_{\ell n}/2}} $ & 
$\alpha_9 = 
\frac{16 \pi^2 }{y_{\ell n}^2} $ & 
$u_9 = 
\frac{\cos(4 \pi y/y_{\ell n})}
{\sqrt{n x_{\ell n} y_{\ell n}/4}} $  \\ \hline
$\alpha_{10} = 
\frac{4 \pi^2 }{n^2x_{\ell n}^2}+\frac{4 \pi^2 }{y_{\ell n}^2} $ & 
$u_{10} = \frac{\sin(
\frac{2 \pi x}{nx_{\ell n}}-\frac{2 \pi y}{y_{\ell n}})}
{\sqrt{n x_{\ell n} y_{\ell n}/2}} $ & 
$\alpha_{10} = 
\frac{64 \pi^2 }{n^2x_{\ell n}^2} $ & 
$u_{10} = \frac{\sin(8 \pi x/(nx_{\ell n}))}
{\sqrt{n x_{\ell n} y_{\ell n}/4}} $  \\ \hline
$\alpha_{11} = 
\frac{4 \pi^2 }{n^2x_{\ell n}^2}+\frac{4 \pi^2 }{y_{\ell n}^2} $ & 
$u_{11} = \frac{\cos(
\frac{2 \pi x}{nx_{\ell n}}-\frac{2 \pi y}{y_{\ell n}})}
{\sqrt{n x_{\ell n} y_{\ell n}/2}} $ & 
$\alpha_{11} = 
\frac{64 \pi^2 }{n^2x_{\ell n}^2} $ & 
$u_{11} = \frac{\cos(8 \pi x/(nx_{\ell n}))}
{\sqrt{n x_{\ell n} y_{\ell n}/4}} $  \\ \hline
$\alpha_{12} = 
\frac{16 \pi^2}{y_{\ell n}^2} $ & 
$u_{12} = 
\frac{\sin(4 \pi y/y_{\ell n})}{\sqrt{n 
x_{\ell n} y_{\ell n}/2}} $ & 
$\alpha_{12} = 
\frac{36 \pi^2 }{n^2x_{\ell n}^2}+\frac{4 \pi^2 }{y_{\ell n}^2} $ & 
$u_{12} = \frac{\sin(
\frac{6 \pi x}{nx_{\ell n}}+\frac{2 \pi y}{y_{\ell n}})}
{\sqrt{n x_{\ell n} y_{\ell n}/4}} $  \\ \hline
$\alpha_{13} = 
\frac{16 \pi^2}{y_{\ell n}^2} $ & 
$u_{13} = 
\frac{\cos(4 \pi y/y_{\ell n})}{\sqrt{n x_{\ell n} 
y_{\ell n}/2}} $ & 
$\alpha_{13} = 
\frac{36 \pi^2 }{n^2x_{\ell n}^2}+\frac{4 \pi^2 }{y_{\ell n}^2} $ & 
$u_{13} = \frac{\cos(
\frac{6 \pi x}{nx_{\ell n}}+\frac{2 \pi y}{y_{\ell n}})}
{\sqrt{n x_{\ell n} y_{\ell n}/4}} $  \end{tabular}
    \end{center}
\caption{The first $13$ eigenvalues 
and eigenfunctions of $-\triangle u_i = \alpha_i u_i$.}
\end{table}

\section{Application of Courant's nodal domain theorem}

In this section, we consider a Jacobi function on each surface 
${\cal W}_{\ell/n}$ produced 
from rotation about the surface's central axis.  Using the 
geometry of the surface, we can estimate the number of nodal 
domains of this function, and then the Courant nodal domain theorem 
gives the following lower bound for the index:  

\begin{lemma}
Ind(${\cal W}_{\ell/n}$) $\geq 2n-2$ if $\ell$ is odd, and 
Ind(${\cal W}_{\ell/n}$) $\geq n-2$ if $\ell$ is even.  
\end{lemma}

\begin{proof}
Let $\vec{v}$ be the variational vector field of 
$\bfR^3$ associated to rotation about the axis of rotational 
symmetry of ${\cal W}_{\ell/n}$.  The normal part 
$\langle \vec{v},\vec{N} \rangle_{\bfR^3} \vec{N}$ is a 
Jacobi field on ${\cal W}_{\ell/n}$, that is, 
$v := \langle \vec{v},\vec{N} \rangle_{\bfR^3}$ satisfies 
${\cal L}(v) = 0$.  Letting $k$ be the number of negative 
eigenvalues of ${\cal L}$, we may assume $v= v_{k+1}$.  

For odd $\ell$ (resp. even $\ell$), the set 
$\{[x,y] \in \bfC/\Gamma \, | \, 2x/x_{\ell n} \in \bfZ \}$ is 
mapped by $\cal X$ to $2n$ (resp. $n$) closed planar geodesics 
in ${\cal W}_{\ell/n}$, each 
lying in a plane containing the axis of rotational symmetry 
of ${\cal W}_{\ell/n}$.  Thus $v_{k+1}$ is zero along these 
$2n$ (resp. $n$) curves, so $v_{k+1}$ has at least 
$2n$ (resp. $n$) nodal domains.  

Courant's nodal domain theorem, which is valid in our 
setting (for example, the proof in \cite{C} easily extends 
to a proof of this theorem for our operator ${\cal L}$), says 
that the number of nodal 
domains of $v_j$ is at most $j$ for all
$j \in\bfZ^+$.  Choosing $j = k+1$, we conclude that 
$2n \leq k+1$ if $\ell$ is odd, and $n \leq k+1$ if 
$\ell$ is even.  Hence there are at least $2n-1$ (resp. $n-1$) 
negative eigenvalues of ${\cal L}$ if $\ell$ is odd (resp. even).  
Lemma 2.4 then implies the result.  
\end{proof}

We immediately have the following corollary: 

\begin{corollary}
For every $N \in \bfZ$, only finitely many 
${\cal W}_{\ell/n}$ satisfy Ind(${\cal W}_{\ell/n}$)$\leq N$.  
\end{corollary}

We remark that Courant's nodal domain theorem can be 
similarly applied to produce lower bounds for the index 
of many of the 
surfaces described in \cite{GBP}, \cite{HL}, and \cite{K1}.  

\section{Bounds for index derived from bounds for $V$}

As in the previous section, we again search for rough bounds for the 
index of the surfaces ${\cal W}_{\ell/n}$.  But instead of using a Jacobi 
function, this time we will use upper and lower bounds for the potential 
function, giving us both upper and lower bounds for the index.  When 
the potential function is replaced by a constant upper or lower bound, 
we know the eigenvalues of the new resulting operator explicitly, leading to 
the lemma below.  (This idea can be found in \cite{FS}, equation (15) of 
Chapter 6.)  

Let $\alpha_{\rho(1)}, \alpha_{\rho(2)}, ... $ be the complete 
set of eigenvalues with multiplicity $1$ of the operator
$- \triangle$ on the flat torus $\bfC/\Gamma$, defined as in 
section 2.3, but reordered by the permutation $\rho(i)$ so that 
$\alpha_{\rho(1)} \leq \alpha_{\rho(2)}
\leq ... \nearrow + \infty$.  (We use the reordering 
$\rho(i)$ only in this section.)  Define the constants 
\[ V_{min} = \min_{[x,y] \in \bfC/\Gamma} V(x,y) \; 
, \; \; \; 
V_{max} = \max_{[x,y] \in \bfC/\Gamma} V(x,y) \; . \]
Consider the operators ${\cal L}_{min} = - \triangle 
- V_{min}$ and 
${\cal L}_{max} = - \triangle - V_{max}$.  The 
complete set of eigenvalues with multiplicity $1$ of 
${\cal L}_{min}$ (resp. ${\cal L}_{max}$) is then 
$\underline{\alpha}_{\rho(i)} = \alpha_{\rho(i)} - V_{min}$ (resp. 
$\overline{\alpha}_{\rho(i)} = \alpha_{\rho(i)} - V_{max}$).  

Recall that $\beta_i$ is a complete set of eigenvalues with 
multiplicity $1$ of ${\cal L}$.  Comparing (2.3) with 
similar variational characterizations for 
the eigenvalues $\underline{\alpha}_{\rho(i)}$ and 
$\overline{\alpha}_{\rho(i)}$, we have 
\[ \underline{\alpha}_{\rho(i)} = \alpha_{\rho(i)} - V_{min} 
\geq \beta_i \geq \overline{\alpha}_{\rho(i)} = \alpha_{\rho(i)} - 
V_{max} \; \; \forall i \; 
. \]  By Lemma 2.4, we conclude:  

\begin{lemma}  
Choose $\mu \in \bfZ^+$ (resp. $\nu \in 
\bfZ^+$) so that $\alpha_{\rho(\mu)} < V_{min}$ (resp. 
$\alpha_{\rho(\nu)} < V_{max}$) and 
$\alpha_{\rho(\mu+1)} \geq V_{min}$ (resp. 
$\alpha_{\rho(\nu+1)} \geq V_{max}$).  Then 
\[ \mu - 1 \leq \mbox{Ind}({\cal W}_{\ell/n}) \leq \nu \; . \] 
\end{lemma} 

Lemma 4.1 gives explicit upper and
lower bounds for the index, since we know the 
$\alpha_{\rho(i)}$ explicitly (see 
Table 2).  Since $V_{max}$ is generally much larger than 
$V_{min}$, these bounds are very rough.  
However, when $\ell/n$ is close to $1$, the bounds get somewhat 
sharp, in the sense that the ratio of the lower and upper 
bounds is close to $1$ (see, for example, the bounds for 
${\cal W}_{21/20}$ and ${\cal W}_{73/72}$ in Table 2).  
We produce much sharper estimates in the next sections.  

\section{Specific spaces on which ${\cal L}$ is negative definite}

We now begin considering the Jacobi operator restricted to finite 
dimensional subspaces.  The essential point is that if the operator 
restricted to some $N$-dimensional subspace is negative definite, then the 
index of ${\cal W}_{\ell/n}$ is at least $N-1$.  We will use subspaces 
generated by 
a finite number of eigenfunctions of the Laplacian.  This will lead us 
to stronger lower bounds for the index.  

Consider a finite subset $\{\tilde{u}_1=u_{i_1},...,
\tilde{u}_N=u_{i_N}\}$ of the eigenfunctions $u_i$ of 
$- \triangle$ on 
$\bfC / \Gamma$ defined in section 2.3 with corresponding 
eigenvalues $\tilde{\alpha}_j = \alpha_{i_j}, \; j=1,...,N$.  
Consider any $u = \sum_{i=1}^N a_i \tilde{u}_i \in 
\mbox{span}\{\tilde{u}_1,...,\tilde{u}_N\}$, $a_1,...,a_N \in 
\bfR$.  
\[ \int_{\bfC / \Gamma} u {\cal L}u dxdy = 
\sum_{i,j = 1}^N a_i a_j \int_{\bfC / \Gamma} 
\tilde{u}_i (-\triangle \tilde{u}_j - V \tilde{u}_j) dxdy 
= \sum_{i,j = 1}^N a_i (\tilde{\alpha}_j 
\delta_{ij} - \tilde{b}_{ij}) a_j \; , 
\] where $\tilde{b}_{ij} := \int_{\bfC / \Gamma} 
V \tilde{u}_i \tilde{u}_j dxdy$.  So $\int_{\bfC/\Gamma} 
u {\cal L}u dxdy < 0$ for all 
nonzero $u \in \mbox{span}\{\tilde{u}_1,...,\tilde{u}_N\}$ if 
and only if 
the matrix $\tilde{A}_N = (\tilde{\alpha}_j \delta_{ij} - 
\tilde{b}_{ij})_{i,j = 1,...,N}$ is negative definite.
The remark following Lemma 2.4 then implies the following 
result:  

\begin{theorem}
If the $N \times N$ matrix $(\tilde{\alpha}_j \delta_{ij} - 
\tilde{b}_{ij})_{i,j = 1,...,N}$ is negative definite, then 
\[ \mbox{Ind}({\cal W}_{\ell/n}) \geq N-1 \; . \]
\end{theorem} 

So, for example, consider the symmetric Wente torus 
${\cal W}_{3/2}$.  Choose $N=9$ and choose 
$\tilde{u}_1 = u_1$, 
$\tilde{u}_2 = u_2$, $\tilde{u}_3 = u_3$, 
$\tilde{u}_4 = u_4$, $\tilde{u}_5 = u_5$, 
$\tilde{u}_6 = u_7$, $\tilde{u}_7 = u_8$, 
$\tilde{u}_8 = u_9$, and $\tilde{u}_9 = u_{17}$.  
Then, with $H = 1/2$, the matrix $\tilde{A}_N$ is 
the following $9 \times 9$ negative definite matrix:  
\[ \tilde{A}_N \approx \left(
\begin{array}{ccccccccc}
-9.50&0&0&0&0&0&0&0&0 \vspace{-0.05in} \\ 
0&-7.99&0&0&0&0&0&0&0 \vspace{-0.05in} \\
0&0&-7.99&0&0&0&0&0&0 \vspace{-0.05in} \\
0&0&0&-1.36&0&0&0&0&0 \vspace{-0.05in} \\
0&0&0&0&-13.2&0&0&0&0 \vspace{-0.05in} \\
0&0&0&0&0&-8.70&0&0&0 \vspace{-0.05in} \\
0&0&0&0&0&0&-5.76&0&0 \vspace{-0.05in} \\
0&0&0&0&0&0&0&-5.76&0 \vspace{-0.05in} \\
0&0&0&0&0&0&0&0&-5.50
\end{array}
\right) \; , \]
Hence, by Theorem 5.1, 
\[ \mbox{Ind}({\cal W}_{3/2}) \geq 8 \; . \]
This is the best we can do for the $\ell/n=3/2$ case; any $10 \times 
10$ matrix of the form $(\tilde{\alpha}_i \delta_{ij} - 
\tilde{b}_{ij})_{i,j = 1,...,10}$ will not be negative definite.  

For ${\cal W}_{4/3}$ we use $N=10$ and the functions 
$u_1$,...,$u_9$,$u_{13}$ for 
$\tilde{u}_1,...,\tilde{u}_{10}$, producing the negative 
definite matrix (with $H = 1/2$) 
\[ \tilde{A}_N \hspace{-0.014in} \approx \hspace{-0.02in} \left(
\begin{array}{cccccccccc}
-5.17&0&0&0&0&0&0&0&-3.23&0 \vspace{-0.05in} \\
0&-3.53&0&0&0&0&0&0&0&0  \vspace{-0.05in} \\
0&0&-3.53&0&0&0&0&0&0&0  \vspace{-0.05in} \\
0&0&0&-3.78&0&-2.29&0&0&0&0  \vspace{-0.05in} \\
0&0&0&0&-3.78&0&-2.29&0&0&0  \vspace{-0.05in} \\
0&0&0&-2.29&0&-3.78&0&0&0&0  \vspace{-0.05in} \\
0&0&0&0&-2.29&0&-3.78&0&0&0  \vspace{-0.05in} \\
0&0&0&0&0&0&0&-0.25&0&0  \vspace{-0.05in} \\
-3.23&0&0&0&0&0&0&0&-2.21&0 \vspace{-0.05in} \\
0&0&0&0&0&0&0&0&0&-1.97
\end{array}
\right) \; , \] Thus 
\[ \mbox{Ind}({\cal W}_{4/3}) \geq 9 \; . \] 

\begin{table}[htb]
    \begin{center}
    \begin{tabular}{l|c|c|c|c|c|c|c|c}
      & & & & & & Lemma & Lemma & Lemma \\
      & & & & & & 3.1's lower & 4.1's lower & 4.1's upper \\
     ${\cal W}_{\ell/n}$ &$\theta$ & $x_{\ell n}$ & $y_{\ell n}$ & 
$V_{min}$ & $V_{max}$ &  bound for & bound for & bound for \\
  &  &  & 
 &   &  & Ind(${\cal W}_{\ell/n}$)   & 
  Ind(${\cal W}_{\ell/n}$)   & Ind(${\cal W}_{\ell/n}$)  \\
\hline
      ${\cal W}_{3/2}$ & $17.7324$ & $2.56$ &$4.21$&$2$ &$123.447$ &$2$&
           $2$        &$213$ \\ \hline
      ${\cal W}_{4/3}$ & $12.7898$  & $3.28$&$6.34$& $2$  &$33.0184$ &$1$& 
            $6$       &$81$ \\ \hline
      ${\cal W}_{5/3}$ & $21.4807$  & $1.76$&$2.64$&  $2$ &$680.157$ &$4$& 
           $2$        &$743$ \\ \hline
      ${\cal W}_{5/4}$ & $9.9285$  & $3.60$ &$7.90$& $2$ & $17.9591$&$6$& 
           $16$        &$169$ \\ \hline
      ${\cal W}_{7/4}$ & $22.8449$  & $1.33$ &$1.94$& $2$ &$2200.7$ &$6$& 
           $2$        &$1815$ \\ \hline 
      ${\cal W}_{6/5}$ & $8.0983$  & $3.79$&$9.18$&  $2$  &$12.4273$ &$3$& 
           $14$        &$83$ \\ \hline 
      ${\cal W}_{7/5}$ & $14.8978$  &$3.00$ &$5.38$& $2$    &$54.5652$ &$8$& 
           $12$        &$351$ \\ \hline 
      ${\cal W}_{8/5}$ & $20.1374$  &$2.08$ &$3.23$& $2$    & $319.339$&$3$& 
           $2$        &$433$ \\ \hline 
      ${\cal W}_{9/5}$ & $23.4867$  &$1.07$ &$1.54$& $2$    &$5426.0$ &$8$& 
            $2$       &$3569$ \\ \hline 
      ${\cal W}_{7/6}$ & $6.8332$  &$3.91$ &$10.3$& $2$  &$9.65614$ &$10$& 
            $38$       &$189$ \\ \hline 
      ${\cal W}_{11/6}$ & $23.8382$ &$0.89$ &$1.28$ & $2$  &$11308.6$ &$10$& 
            $2$       &$6191$ \\ \hline 
      ${\cal W}_{8/7}$ & $5.9081$  &$3.99$ &$11.3$&  $2$ &$8.01729$ &$5$& 
            $24$       &$97$ \\ \hline 
      ${\cal W}_{9/7}$ & $11.1844$  &$3.47$ &$7.16$&  $2$  &$23.2959$ &$12$& 
            $28$       &$323$ \\ \hline 
      ${\cal W}_{10/7}$ & $15.7491$  & $2.88$&$5.02$& $2$   &$68.2461$ &$5$& 
            $8$       &$277$ \\ \hline 
      ${\cal W}_{11/7}$ & $19.4966$  &$2.22$ &$3.50$& $2$   &$238.928$ &$12$& 
            $6$       &$1037$ \\ \hline 
      ${\cal W}_{12/7}$ & $22.3044$  &$1.52$ &$2.24$&  $2$  &$1278.61$ &$5$& 
            $2$       &$1201$ \\ \hline 
      ${\cal W}_{13/7}$ & $24.0512$  &$0.77$ &$1.10$&  $2$  &$21012.8$ &$12$& 
            $2$       &$9863$ \\ \hline
      ${\cal W}_{21/20}$ & $2.1359$  &$4.29$ &$20.2$&  $2$  &$3.54102$ &$38$& 
            $278$       &$491$ \\ \hline
      ${\cal W}_{73/72}$ & $0.6005$  &$4.40$ &$39.1$&  $2$  &$2.3828$ &$142$& 
             $1962$      & $2353$
    \end{tabular}
    \end{center}
\caption{Rough upper and lower bounds for Ind(${\cal W}_{\ell/n}$).  
($x_{\ell n}$, $y_{\ell n}$, $V_{min}$, and $V_{max}$ are computed using 
the value $H=1/2$.)}
\end{table}

For other surfaces we can use the same technique.  Without showing 
the resulting matrices, we provide here some other examples: 
\begin{center}
For ${\cal W}_{5/3}$ we can use $N=12$ and the functions 
$u_1$,$u_2$,$u_3$,$u_5$,...,$u_9$,$u_{15}$,$u_{16}$,$u_{17}$,$u_{29}$.  
\end{center}
\begin{center}
For ${\cal W}_{5/4}$ we can use $N=33$ and 
$u_1$,...,$u_{23}$,$u_{27}$,...,$u_{35}$,$u_{45}$.  
\end{center}
\begin{center}
For ${\cal W}_{7/4}$ we can use $N=16$ and 
$u_1$,$u_2$,$u_3$,$u_5$,...,$u_9$,$u_{14}$,...,$u_{17}$,$u_{27}
$,$u_{28}$,$u_{29}$,$u_{45}$.  
\end{center}
\begin{center}
For ${\cal W}_{6/5}$ we can use $N=20$ and 
$u_1$,...,$u_{19}$,$u_{29}$.  
\end{center}
\begin{center}
For ${\cal W}_{7/5}$ we can use $N=27$ and 
$u_1$,...,$u_{11}$,$u_{14}$,...,$u_{19}$,$u_{26}$,...,$u_{31}
$,$u_{43}$,$u_{44}$,$u_{45}$,$u_{65}$.  
\end{center}
\begin{center}
For ${\cal W}_{8/5}$ we can use $N=12$ and 
$u_1$,...,$u_7$,$u_{10}$,...,$u_{13}$,$u_{29}$.  
\end{center}
\begin{center}
For ${\cal W}_{9/5}$ we can use $N=20$ and 
$u_1$,$u_2$,$u_3$,$u_5$,...,$u_9$,$u_{14}$,...,$u_{17}$,$u_{26}
$,...,$u_{29}$,$u_{43}$,$u_{44}$,$u_{45}$,$u_{65}$.  
\end{center}
In each case, 
$\tilde{A}_N$ will be negative definite.  $\tilde{A}_N$ is 
a diagonal matrix for ${\cal W}_{3/2}$, 
${\cal W}_{5/3}$, ${\cal W}_{7/4}$, and 
${\cal W}_{9/5}$, and is not diagonal for 
${\cal W}_{4/3}$, ${\cal W}_{5/4}$, ${\cal W}_{6/5}$, 
${\cal W}_{7/5}$, and ${\cal W}_{8/5}$.  We conclude that 
\[ \mbox{Ind}({\cal W}_{5/3}) \geq 11 \; , \; \; \; 
\mbox{Ind}({\cal W}_{5/4}) \geq 32 \; , \; \; \; 
\mbox{Ind}({\cal W}_{7/4}) \geq 15 \; , \; \; \; 
\mbox{Ind}({\cal W}_{6/5}) \geq 19 \; , \] 
\[ \mbox{Ind}({\cal W}_{7/5}) \geq 26 \; , \; \; \; 
\mbox{Ind}({\cal W}_{8/5}) \geq 11 \; , \; \; \; 
\mbox{Ind}({\cal W}_{9/5}) \geq 19 \; . \] 
We could continue in this way for more ${\cal W}_{\ell/n}$, 
but many computations are required to find sharp lower 
bounds.  However, if we are interested in rougher 
lower bounds, then we may continue without 
so many computations.  For example, for each of the surfaces 
${\cal W}_{8/7}$, ${\cal W}_{10/7}$, and ${\cal W}_{12/7}$, 
we can use $N=9$ and the functions 
$u_1,...,u_5,u_{10},...,u_{13}$ to 
create a negative definite diagonal matrix $\tilde{A}_N$, thus 
\[ \mbox{Ind}({\cal W}_{8/7}) \geq 8 \; , \; \; \; 
\mbox{Ind}({\cal W}_{10/7}) \geq 8 \; , \; \; \; 
\mbox{Ind}({\cal W}_{12/7}) \geq 8 \; . \] 
From Lemma 3.1 and the above applications of Theorem 5.1, we 
conclude:

\begin{center}
{\em Ind(${\cal W}_{\ell/n}$)$\geq 7$ for all ${\cal W}_{\ell/n}$.}  
\end{center}

\begin{remark}
In fact, numerical results in section 7 suggest that 
Ind(${\cal W}_{\ell/n}$)$\geq 9$ for all ${\cal W}_{\ell/n}$.  
Since the symmetric Wente tori are possible candidates for 
minimizers of index amongst all closed unstable 
CMC surfaces, one may conjecture that, except 
for the round sphere, there does not exist a closed 
CMC surface in $\bfR^3$ with index less than $9$.  
\end{remark}

\section{A method for computing the index sharply}

In the previous section, we found finite dimensional 
subspaces on which the restricted 
operator is negative definite.  Now, we will still use finite dimensional 
subspaces, but we will no longer confine ourselves to subspaces for which 
the operator becomes negative definite.  Instead we will use the 
number of negative eigenvalues of the restricted operators on the subspaces 
to estimate the index of the surfaces from below.  
Again, as they are convenient to work with, we will use subspaces generated 
by eigenfunctions of the Laplacian.  Furthermore, we will show that if 
we choose a sequence of subspaces whose union is dense in the full space, 
then the eigenvalues of the restricted operators converge to the eigenvalues 
of the 
original operator, providing us with a numerical algorithm for estimating 
the eigenvalues sharply.  

Let $P_m$ be orthogonal projection of $L^2$ to 
${\cal V}_m := \mbox{span}\{u_1,...,
u_m\}$, the span of the first $m$ eigenfunctions $u_i$ (as ordered 
in section 2.3) of 
$-\triangle$.  Hence $P_m(u = \sum_{i=1}^\infty a_i u_i)$ $= 
\sum_{i=1}^m a_i u_i = 
\sum_{i=1}^m \langle u, u_i \rangle_{L^2} u_i $.  
Note that $P_m v_j \to_{L^2} v_j$ as $m \rightarrow \infty$ for any 
eigenfunction $v_j$ of $\cal L$.  

Let ${\cal L}_m = P_m \circ {\cal L} \circ P_m$.  Then 
\[{\cal L}_m \left( u = \sum_{i=1}^\infty a_iu_i \right) = 
P_m \circ \left( \sum_{i=1}^m a_i {\cal L}u_i \right) = 
\sum_{i,j=1}^m a_i (\alpha_i 
\delta_{ij} - b_{ij}) u_j \; , \; \; \; \mbox{where} \]
\[ b_{ij} := \int_{\bfC / \Gamma} V u_i u_j dxdy \; . 
\]  So, on ${\cal V}_m$, 
${\cal L}_m$ is the linear transformation $A_m = 
(\alpha_i \delta_{ij} - b_{ij})_{1 \leq i,j, \leq m}$ with respect to 
the basis $\{u_i\}_{i=1}^m$.  
The next lemma states that the linear map $A_m$ is the 
restriction of ${\cal L}$ to ${\cal V}_m$:  

\begin{lemma}
Let $\hat{x} = (x_1,...,x_m), \hat{y} = (y_1,...,y_m) \in \bfR^m$.  
Let ${\sf f} = \sum_{i=1}^m x_i u_i, {\sf g} = \sum_{i=1}^m y_i u_i$.  
Then 
\[ \langle \hat{x}, A_m \hat{y} \rangle_{\bfR^m} = 
\int_{\bfC/\Gamma} {\sf f}{\cal L}{\sf g} dxdy \; \; 
\mbox{ and } \; \; \langle \hat{x}, \hat{y} \rangle_{\bfR^m} = 
\int_{\bfC/\Gamma} {\sf f}{\sf g} dxdy \; . \] 
\end{lemma}

\begin{proof}
$\int_{\bfC/\Gamma} {\sf f}{\cal L}{\sf g} dxdy = 
\sum_{i,j=1}^m x_iy_j (\alpha_i \delta_{ij} - b_{ij}) = 
\langle \bar{x}, A_m \bar{y} \rangle_{\bfR^m}$.  
Similarly, $\langle \hat{x}, 
\hat{y} \rangle_{\bfR^m} = \int_{\bfC/\Gamma} {\sf f}{\sf g} dxdy$.  
\end{proof}

Let $\lambda_1^{(m)} \leq ... \leq \lambda_m^{(m)}$ be the 
eigenvalues of $A_m$ with corresponding 
orthonormal eigenvectors $w_1^{(m)},...,
w_m^{(m)}$.  Note that since $A_m$ is a symmetric matrix, 
$\lambda_i^{(m)} \in \bfR$ for all $i$. 
The following lemma is known (\cite{RS}, Theorem VIII.3, or \cite{FS}, 
equation (23) in Chapter 6), but 
we include a brief proof here.  

\begin{lemma}
$\lambda_j^{(m)} \geq \lambda_j^{(m^\prime)} \geq \beta_j$ for 
all $j$, $m$, $m^\prime$ such that $j \leq m \leq m^\prime$.  
\end{lemma}

\begin{proof}
We have the following variational characterization:
\[
\lambda_j^{(m)} = \sup_{\phi_1,...,\phi_{j-1} \in {\cal V}_m} 
\left( \inf_{\psi \in {\cal V}_m \cap 
(\mbox{span}\{\phi_1,...,\phi_{j-1}\})^\perp, || \psi ||_{L^2}=1 } 
\int_{\bfC/\Gamma} \psi {\cal L}_m \psi dxdy \right)
\]
\[
 = \sup_{\phi_1,...,\phi_{j-1} \in L^2} 
\left( \inf_{\psi \in {\cal V}_m \cap 
(\mbox{span}\{P_m(\phi_1),...,P_m(\phi_{j-1})\})^\perp, 
|| \psi ||_{L^2}=1 } \int_{\bfC/\Gamma} 
\psi (P_m \circ {\cal L} \circ P_m (\psi)) dxdy \right) \; .
\]
However, for $\psi \in {\cal V}_m$ we have
$\int_{\bfC/\Gamma} \psi (P_m \circ {\cal L} \circ 
P_m (\psi)) dxdy = 
\int_{\bfC/\Gamma} \psi {\cal L} \psi dxdy$, so 
\[
\lambda_j^{(m)} = \sup_{\phi_1,...,\phi_{j-1} \in L^2} 
\left( \inf_{\psi \in {\cal V}_m \cap 
(\mbox{span}\{\phi_1,...,\phi_{j-1}\})^\perp, || \psi ||_{L^2}=1 } 
\int_{\bfC/\Gamma} \psi {\cal L} \psi dxdy \right) \; .
\] So $\lambda_j^{(m)}$ is 
nonincreasing as $m$ increases.  Also, $\beta_j$ has the following 
characterization \cite{U}:
\[ \beta_j = \sup_{\phi_1,...,\phi_{j-1} \in L^2} 
\left( \inf_{\psi \in L^2 \cap 
(\mbox{span}\{\phi_1,...,\phi_{j-1}\})^\perp, || \psi ||_{L^2}=1 } 
\int_{\bfC/\Gamma} \psi {\cal L} \psi dxdy \right) \; .
\] Hence $\lambda_j^{(m)} \geq \beta_j$ for all $m$.  
\end{proof}

We define 
\[ \gamma_j := \lim_{m \to \infty} \lambda_j^{(m)} \; . \] 
This limit clearly exists, since 
$\{ \lambda_j^{(m)} \}_{m=j}^\infty$ is 
nonincreasing and bounded below by $\beta_j$.  

\begin{theorem}
$\gamma_j = \beta_j$ for all $j$.  
\end{theorem}

\begin{remark}
The motivation for this theorem is that it 
gives us a method for estimating the eigenvalues $\beta_j$ of 
${\cal L}$, since one can estimate $\gamma_j$ numerically.  Thus 
one can estimate the number of negative eigenvalues $\beta_j < 0$, 
and hence the index (see section 7).  
\end{remark}

Before proving Theorem 6.1, we prove a 
crucial preliminary lemma, which essentially 
says that the Rayleigh quotient 
\[ \frac{\int_{\bfC/\Gamma} 
(P_m v_j) {\cal L}(P_m v_j) dxdy}
{\int_{\bfC/\Gamma} (P_m v_j)^2 dxdy} \; \; \; \mbox{ of } 
P_m v_j \mbox{ converges to the} \]
\[ \mbox{Rayleigh quotient } 
\frac{\int_{\bfC/\Gamma} v_j {\cal L} v_j dxdy}
{\int_{\bfC/\Gamma} (v_j)^2 dxdy} = \beta_j \; \; \; 
\mbox{ of } v_j 
\mbox{ as } m \to \infty, \mbox{ for all } j \; . \]

\begin{lemma}
$\int_{\bfC/\Gamma} (P_m v_j) {\cal L}(P_m v_j) dxdy \to 
\int_{\bfC/\Gamma} v_j {\cal L} v_j dxdy = \beta_j$ as 
$m \to \infty$ for all $j$.  
\end{lemma}

\begin{proof}
$\langle - \triangle v_j,v_j 
\rangle_{L^2}$ $=$ $\sum_{i,l=1}^\infty$ $\langle 
-\triangle (
\langle v_j,u_i  \rangle_{L^2} u_i), \langle v_j,u_l 
\rangle_{L^2} u_l \rangle_{L^2}$ $= 
\sum_{i=1}^\infty 
\langle v_j,u_i \rangle_{L^2}^2 \alpha_i$ $< \infty$.  
Similarly, 
$\langle -\triangle P_m v_j, P_m v_j \rangle_{L^2} = 
\sum_{i=1}^m \langle v_j,u_i \rangle_{L^2}^2 \alpha_i$, 
hence 
\[ \int_{\bfC/\Gamma} (P_m v_j) (-\triangle P_m v_j) dxdy = 
\sum_{i=1}^m \langle v_j,u_i \rangle_{L^2}^2 \alpha_i 
\to \sum_{i=1}^\infty 
\langle v_j,u_i \rangle_{L^2}^2 \alpha_i = 
\int_{\bfC/\Gamma} v_j (-\triangle v_j) dxdy \] 
as $m \to \infty$. 

We complete the proof by showing 
$\int_{\bfC / \Gamma} V (P_m v_j)^2 dxdy \to 
\int_{\bfC / \Gamma} V v_j^2 dxdy$ as $m \to \infty$.  
Let $M_m^+ = \{ [x,y] \in \bfC / \Gamma \, 
| \, (P_m v_j)^2 - 
v_j^2 \geq 0 \}$, and let 
$M_m^- = \{ [x,y] \in \bfC / \Gamma \, | \, (P_m 
v_j)^2 - v_j^2 < 0 \}$.  
\[ \int_{\bfC / \Gamma} (P_m v_j - v_j)^2 dxdy \to 0 
\Rightarrow 
\int_{M_m^+} (P_m v_j - v_j)^2 dxdy \to 0 \mbox{ and } 
\int_{M_m^-} (P_m v_j - v_j)^2 dxdy \to 0 \] as 
$m \to \infty$, so 
for $M_m^*$ equal to either $M_m^+$ or 
$M_m^-$, we have
\[ \int_{M_m^*} \hspace{-0.04in} |(P_m v_j)^2 - v_j^2| dxdy = 
 \hspace{-0.04in} 
\left| \int_{M_m^*} ((P_m v_j)^2 
\hspace{-0.04in} -  \hspace{-0.04in} v_j^2) dxdy \right| = 
 \hspace{-0.04in} 
\left| \int_{M_m^*} (P_m v_j \hspace{-0.04in} - \hspace{-0.04in} 
v_j)^2 \hspace{-0.04in} + \hspace{-0.04in} 2v_j (P_mv_j 
\hspace{-0.04in} - \hspace{-0.04in} v_j) dxdy \right| 
\]\[ \leq \int_{M_m^*} (P_m v_j - v_j)^2 dxdy + 2 || 
v_j ||_{L^2(\bfC/\Gamma)} 
|| P_m v_j - v_j ||_{L^2(\bfC/\Gamma)} \to 0 \mbox{ as } m \to 
\infty \] by the Cauchy-Schwarz inequality.  Hence 
\[ \left| \int_{\bfC/\Gamma} V((P_m v_j)^2 - v_j^2) dxdy 
\right| \leq 
 V_{max} \int_{\bfC/\Gamma} | (P_m v_j)^2 - v_j^2 | dxdy = 
\] \[ V_{max} \int_{M_m^+} ((P_m v_j)^2 - v_j^2) dxdy \; \; - 
\; \; V_{max} \int_{M_m^-} 
((P_m v_j)^2 - v_j^2) dxdy \; \; \; \to 0 
\mbox{ as } m \to \infty \; . \]
\end{proof}

\begin{proofspec}
Since $w_i^{(m)}$ is a linear combination of $u_j$ for 
$j \leq m$, we may consider $w_i^{(m)}$ as a function in 
$C^\infty(\bfC/\Gamma)$.  Then 
$\int_{\bfC/\Gamma} w_\mu^{(m)} {\cal L}w_\nu^{(m)} dxdy = 
\langle w_\mu^{(m)} , A_m w_\nu^{(m)} \rangle_{\bfR^m} = 
\lambda_\nu^{(m)} \delta_{\mu\nu}$ and 
$\int_{\bfC/\Gamma} w_\mu^{(m)} w_\nu^{(m)} dxdy = 
\langle w_\mu^{(m)} , w_\nu^{(m)} \rangle_{\bfR^m} = 
\delta_{\mu\nu}$, by Lemma 6.1.  Choose 
$a_{i,j}^{(m)} \in \bfR$ so that 
$P_m v_j = \sum_{i=1}^m a_{i,j}^{(m)} 
w_i^{(m)}$ for all $m$ and $j$, then 
\begin{equation}
\int_{\bfC/\Gamma} 
(P_mv_j) {\cal L}(P_mv_j) dxdy = \sum_{i=1}^m (a_{i,j}^{(m)})^2 
\lambda_i^{(m)} \to \beta_j 
\end{equation}
by Lemma 6.3, and 
\begin{equation} 
\langle P_mv_j , P_mv_j \rangle = 
\sum_{i=1}^m (a_{i,j}^{(m)})^2 \to 1  
\end{equation}
for all $j \in \bfZ^+$ as $m \to \infty$.  

Let $P_l^m$ be the orthogonal projection from $L^2$ to 
span$\{w_1^{(m)},...,w_l^{(m)}\}$.  

Choose 
$k_1,k_2,...$ such that $\beta_1 = ... = \beta_{k_1} 
< \beta_{k_1+1} = ... = \beta_{k_2} < \beta_{k_2+1} 
= ... = \beta_{k_3} < \beta_{k_3+1} = ... \, $.  Note that 
$k_1 = 1$, since Courant's nodal domain theorem implies 
that any 
function in the eigenspace of the first eigenvalue has only 
one nodal domain, hence this eigenspace cannot contain 
two $L^2$-orthogonal functions.  

We will prove the result by induction on $i$, the subscript 
of $k_i$, for $i \in \bfZ^+$.  

{\bf Base case: $\gamma_1 = \beta_1$, and $P_{1}^m v_1 
\to_{L^2} v_1$ as $m \to \infty$.}  

We have 
\[ 0 \leq (a_{1,1}^{(m)})^2 (\gamma_1-\beta_1) + 
(\sum_{\rho=2}^m (a_{\rho,1}^{(m)})^2) (\beta_2-\beta_1) 
\]\[
\leq (a_{1,1}^{(m)})^2 (\gamma_1-\beta_1) + 
((\sum_{\rho=1}^m (a_{\rho,1}^{(m)})^2) -1)\beta_1 + 
|(\sum_{\rho=1}^m (a_{\rho,1}^{(m)})^2) -1| \cdot |\beta_1| + 
(\sum_{\rho=2}^m (a_{\rho,1}^{(m)})^2) (\beta_{2}-\beta_1) 
\]\[
= (a_{1,1}^{(m)})^2 \gamma_1 - \beta_1 + 
(\sum_{\rho=2}^m (a_{\rho,1}^{(m)})^2) \beta_{2} + 
|(\sum_{\rho=1}^m (a_{\rho,1}^{(m)})^2) -1| \cdot |\beta_1|
\]\[
\leq (\sum_{\rho=1}^m (a_{\rho,1}^{(m)})^2 \lambda_\rho^{(m)}) - \beta_1 + 
|(\sum_{\rho=1}^m (a_{\rho,1}^{(m)})^2) -1| \cdot |\beta_1|
\to 0 \] as $m \to \infty$, by 
equations (6.4) and (6.5).  Therefore, since 
$\beta_{2} - \beta_1 > 0$, we have 
\[ || P_{1}^m v_1 
- P_mv_1 ||_{L^2}^2 = \sum_{\rho=2}^m (a_{\rho,1}^{(m)})^2 \to 0 \; , 
\]\[ (a_{1,1}^{(m)})^2 \to 1 \; , \mbox{ and } 
(a_{1,1}^{(m)})^2 (\gamma_1-\beta_1) \to 0 \mbox{ as } 
m \to \infty \; . \]
Thus $\gamma_1 = \beta_1$, and 
$|| P_{1}^m v_1 - v_1 ||_{L^2} \leq 
|| P_{1}^m v_1 - P_mv_1 ||_{L^2} + 
|| P_m v_1 - v_1 ||_{L^2} \to 0$ as $m \to \infty$.  

{\bf Inductive step: If 
$\gamma_j = \beta_j$ and $P_{k_i}^m v_j 
\to_{L^2} v_j$ as $m \to \infty$ for all $j \leq k_i$, then 
$\gamma_j = \beta_j$ and $P_{k_{i+1}}^m v_j 
\to_{L^2} v_j$ as $m \to \infty$ for all $j \leq k_{i+1}$.}  

Since $P_{k_i}^mv_j \to_{L^2} v_j$ as $m \to 
\infty$ for all $j \leq k_i$, we have 
$P_{k_i}^mv_j \to_{L^2} 0$ for all 
$j \in [k_i+1,k_{i+1}]$ as $m \to 
\infty$ (since the $v_j$ are orthonormal), so 
\begin{equation} 
a_{l,j}^{(m)} \to 0 \mbox{ for all } 
l \leq k_i \mbox{ and all } j \in [k_i+1,k_{i+1}] \; . 
\end{equation}
Then for all $j \in [k_i+1,k_{i+1}]$, we have 
\[ 0 \leq (\sum_{\rho={k_i+1}}^{k_{i+1}} (a_{\rho,j}^{(m)})^2 
(\gamma_{\rho}-\beta_{\rho}) ) + 
(\sum_{\rho={k_{i+1}+1}}^{m} (a_{\rho,j}^{(m)})^2) 
(\beta_{k_{i+1}+1}-\beta_{k_i+1}) \]\[
\leq (\sum_{\rho={1}}^{k_{i+1}} (a_{\rho,j}^{(m)})^2 (\gamma_\rho-
\beta_\rho)) + 
(\sum_{\rho={k_{i+1}+1}}^{m} (a_{\rho,j}^{(m)})^2)
(\beta_{k_{i+1}+1}-\beta_{k_i+1}) + \]\[ 
(\sum_{\rho=1}^{k_i}(a_{\rho,j}^{(m)})^2 \beta_\rho) + 
((\sum_{\rho=k_i+1}^{m} (a_{\rho,j}^{(m)})^2) - 1)\beta_{k_i+1} + 
(\sum_{\rho=1}^{k_i} (a_{\rho,j}^{(m)})^2 |\beta_\rho|) + 
\]\[ |(\sum_{\rho=k_i+1}^{m} (a_{\rho,j}^{(m)})^2) -1| \cdot |\beta_{k_i+1}| 
= (\sum_{\rho=1}^{k_{i+1}} (a_{\rho,j}^{(m)})^2 \gamma_\rho)-\beta_{k_i+1} + 
\]\[ (\sum_{\rho=k_{i+1}+1}^{m} (a_{\rho,j}^{(m)})^2) \beta_{k_{i+1}+1} + 
(\sum_{\rho=1}^{k_i} (a_{\rho,j}^{(m)})^2 |\beta_\rho|) + 
\]\[ |(\sum_{\rho=k_i+1}^{m} (a_{\rho,j}^{(m)})^2) -1| \cdot |\beta_{k_i+1}| 
\leq (\sum_{\rho=1}^{m} (a_{\rho,j}^{(m)})^2 \lambda_\rho^{(m)}) - 
\beta_{k_i+1} + \]\[ 
(\sum_{\rho=1}^{k_i} (a_{\rho,j}^{(m)})^2 |\beta_\rho|) + 
|(\sum_{\rho=k_i+1}^{m} (a_{\rho,j}^{(m)})^2) -1| \cdot |\beta_{k_i+1}| 
\to 0 \] as $m \to \infty$, by (6.4), (6.5), and 
(6.6).  Therefore, since 
$\beta_{k_{i+1}+1} - \beta_{k_i+1} > 0$, we have 
\begin{equation}
\sum_{\rho=k_i+1}^{k_{i+1}} (a_{\rho,j}^{(m)})^2 (\gamma_{\rho}-\beta_{\rho}) 
\to 0 
\end{equation} and 
\begin{equation} 
|| P_{k_{i+1}}^m v_j - P_m v_j ||_{L^2}^2 = 
\sum_{\rho=k_{i+1}+1}^{m} (a_{\rho,j}^{(m)})^2 \to 0 
\end{equation} as $m \to \infty$.  

By (6.8) we have 
$|| P_{k_{i+1}}^m v_j - v_j ||_{L^2} \leq 
|| P_{k_{i+1}}^m v_j - P_m v_j ||_{L^2} + 
|| P_m v_j - v_j ||_{L^2} \to 0$ as $m \to \infty$, 
hence $P_{k_{i+1}}^m v_j \to_{L^2} v_j$ as 
$m \to \infty$ for all $j \in [k_i+1,k_{i+1}]$.  
Since, by the inductive assumption, 
$P_{k_{i}}^m v_j \to_{L^2} v_j$ as 
$m \to \infty$ for all $j \leq k_{i}$, it 
follows that also 
$P_{k_{i+1}}^m v_j \to_{L^2} v_j$ as 
$m \to \infty$ for all $j \leq k_i$.  
Hence $P_{k_{i+1}}^m v_j\to_{L^2} v_j$ as 
$m \to \infty$ for all $j \leq k_{i+1}$.  

Since $P_{k_{i+1}}^m v_j \to_{L^2} v_j$ as 
$m \to \infty$ for all $j \leq k_{i+1}$, we have 
$\langle P_{k_{i+1}}^m v_\mu,P_{k_{i+1}}^m v_\nu 
\rangle_{L^2} \to \delta_{\mu\nu}$ as $m \to \infty$ for 
all $\mu,\nu \leq 
k_{i+1}$.  By (6.6), $\sum_{l=k_i+1}^{k_{i+1}} 
a_{l,\mu}^{(m)} a_{l,\nu}^{(m)} = 
\langle P_{k_{i+1}}^m v_\mu, P_{k_{i+1}}^m v_\nu 
\rangle_{L^2} - \sum_{l=1}^{k_{i}} 
a_{l,\mu}^{(m)} a_{l,\nu}^{(m)} \to \delta_{\mu\nu}$ as 
$m \to \infty$ when $\mu,\nu \in [k_i+1,k_{i+1}]$.  
So the matrix 
$(a_{\mu,\nu}^{(m)})_{k_i+1 \leq \mu,\nu \leq k_{i+1}}$ is 
almost orthogonal for large $m$.  This implies, 
as (6.7) holds for all $j \in [k_i+1,k_{i+1}]$, that 
$\gamma_{k_i+1} = ... = \gamma_{k_{i+1}} = \beta_{k_i+1}$.  
\end{proofspec}

\section{Numerical computation of the index}

\begin{table}[htb]
    \begin{center}
    \begin{tabular}{l|c|c|c|c|c}
      & Theorem 5.1's& Theorem 6.1's & &range of & range of 
       $A_m$'s \\
        ${\cal W}_{\ell/n}$  & lower bound 
     &   estimate  & $A_m$ & $A_m$'s negative & 
           first 6 positive\\
   &  for 
  Ind(${\cal W}_{\ell/n}$) &  for 
  Ind(${\cal W}_{\ell/n}$) & &eigenvalues & eigenvalues \\
\hline
      ${\cal W}_{3/2}$ &$8$& $10$ or $11$ &$A_{181}$ &
              $(-35.4,-1.47)$     & $(0.059,0.65)$    \\ \hline
      ${\cal W}_{4/3}$ &$9$ & $9$ or $10$ & $A_{81}$ &
               $(-10.2,-0.63)$    & $(0.26,0.47)$    \\ \hline
      ${\cal W}_{5/3}$ & $11$& $11$ or $12$ &$A_{85}$ &
              $(-81.3,-50.6)$         & $(0.093,5.75)$    \\ \hline
      ${\cal W}_{5/4}$ & $32$& $33$ or $34$  &$A_{145}$ &
              $(-5.78,-0.035)$         & $(0.21,0.32)$    \\ \hline
      ${\cal W}_{7/4}$ &$15$ & $15$ or $16$ &$A_{145}$ &
              $(-193.7,-108.5)$         & $(3.88,11.3)$    \\ \hline 
      ${\cal W}_{6/5}$ &$19$ & $19$ or $20$ &$A_{81}$ &
               $(-4.37,-0.76)$        & $(0.12,0.19)$    \\ \hline 
      ${\cal W}_{7/5}$ &$26$ & $26$ or $27$ &$A_{145}$ &
               $(-12.9,-0.72)$        & $(0.039,0.42)$    \\ \hline 
      ${\cal W}_{8/5}$ &$11$ & $11$ or $12$ &$A_{81}$ &
               $(-48.9,-0.89)$        & $(1.83,4.21)$    \\ \hline 
      ${\cal W}_{9/5}$ &$19$ & $19$ or $20$ &$A_{145}$ &
               $(-294.3,-169.7)$        & $(9.38,21.1)$    \\ \hline 
      ${\cal W}_{7/6}$ &$--$ & $53$ or $54$ &$A_{145}$ &
              $(-3.82,-0.068)$         & $(0.11,0.37)$    \\ \hline 
      ${\cal W}_{11/6}$& $--$ & $23$ or $24$ &$A_{181}$ &
              $(-445.3,-254.5)$         & $(10.4,21.9)$    \\ \hline 
      ${\cal W}_{8/7}$ &$8$ & $34$ or $35$ &$A_{81}$ &
               $(-3.48,-0.024)$        & $(0.043,0.38)$    \\ \hline 
      ${\cal W}_{9/7}$ &$--$ & $53$ or $54$ &$A_{145}$ &
              $(-6.24,-0.074)$         & $(0.092,0.39)$    \\ \hline 
      ${\cal W}_{10/7}$ &$8$ & $17$ or $18$ &$A_{81}$ &
              $(-12.1,-1.29)$         & $(0.33,0.83)$    \\ \hline 
      ${\cal W}_{11/7}$ &$--$ & $34$ or $35$ &$A_{145}$ &
              $(-26.9,-0.037)$         & $(1.28,3.10)$    \\ \hline 
      ${\cal W}_{12/7}$ &$8$ & $13$ or $14$ &$A_{81}$ &
              $(-78.3,-57.4)$         & $(2.08,10.6)$    \\ \hline 
      ${\cal W}_{13/7}$ &$--$ & $27$ or $28$ &$A_{181}$ &
              $(-503.0,-278.3)$         & $(24.8,37.3)$    
    \end{tabular}
    \end{center}
\caption{Sharper lower bounds and numerical estimates of 
Ind(${\cal W}_{\ell/n}$).  (The values in the final two columns 
are computed using $H = 1/2$.)}
\end{table}

At first glance, the numerical 
computation of $A_m$ seems to involve a separate 
computation for each $b_{ij}$ for $1 \leq i \leq j \leq m$.  
Each $b_{ij}$ is essentially a triple integral, since 
$V$ is defined via the Jacobi $cn$ function, and $cn$ is 
defined via an integral.  Numerically estimating 
$m(m+1)/2$ triple integrals $b_{ij}$ would require much 
computer time.
However, most of the $b_{ij}$ are $0$, and there 
are clear relationships between many of the 
nonzero $b_{ij}$.  So one wishes to find a 
minimal set of integrals that will 
determine all $b_{ij}$, hopefully having much less than 
$m(m+1)/2$ integrals, thus reducing the amount of required 
numerical computation.  This is the purpose of this section.  

When $\ell$ is odd, we have 
\[ b_{11} = \frac{1}{nx_{\ell n}y_{\ell n}} 
\int_0^{nx_{\ell n}} \int_0^{y_{\ell n}} V(x,y) 
dydx \; , 
\] and 
\[ b_{1j} =  \frac{\sqrt{2}}{nx_{\ell n}y_{\ell n}} 
\int_0^{nx_{\ell n}} \int_0^{y_{\ell n}} V(x,y) 
(\sin \, \mbox{{\scriptsize or}} \, \cos)
(\frac{2\pi a x}{nx_{\ell n}}+\frac{2\pi b y}{y_{\ell n}}) 
dydx \; , 
\] for some $a,b \in \bfZ$ when $j \geq 2$, and 
\[ b_{ij} =  \frac{2}{nx_{\ell n}y_{\ell n}} 
\int_0^{nx_{\ell n}} \int_0^{y_{\ell n}} V(x,y) 
(\sin \, \mbox{{\scriptsize or}} \, \cos)
(\frac{2\pi a x}{nx_{\ell n}}+\frac{2\pi b y}{y_{\ell n}}) 
(\sin \, \mbox{{\scriptsize or}} \, \cos)
(\frac{2\pi c x}{nx_{\ell n}}+\frac{2\pi d y}{y_{\ell n}}) 
dydx \; , 
\] for some $a,b,c,d \in \bfZ$ when $i,j \geq 2$.  
When $\ell$ is even, we have 
\[ b_{11} =  \frac{2}{nx_{\ell n}y_{\ell n}} 
\int_0^{nx_{\ell n}/2} \int_0^{y_{\ell n}} V(x,y) 
dydx \; , 
\] and 
\[ b_{1j} =  \frac{2 \sqrt{2}}{nx_{\ell n}y_{\ell n}} 
\int_0^{nx_{\ell n}/2} \int_0^{y_{\ell n}} V(x,y) 
(\sin \, \mbox{{\scriptsize or}} \, \cos)
(\frac{2\pi a x}{nx_{\ell n}}+\frac{2\pi b y}{y_{\ell n}}) 
dydx \; , 
\] for some $a,b \in \bfZ$ with $a+b$ even 
when $j \geq 2$, and 
\[ b_{ij} \hspace{-0.006in} = \hspace{-0.014in} 
\frac{4}{nx_{\ell n}y_{\ell n}} 
\int_0^{nx_{\ell n}/2} \int_0^{y_{\ell n}} V(x,y) 
(\sin \, \mbox{{\scriptsize or}} \, \cos)
(\frac{2\pi a x}{nx_{\ell n}}+\frac{2\pi b y}{y_{\ell n}}) 
(\sin \, \mbox{{\scriptsize or}} \, \cos)
(\frac{2\pi c x}{nx_{\ell n}}+\frac{2\pi d y}{y_{\ell n}}) 
dydx \; , \] for some 
$a,b,c,d \in \bfZ$ with $a+b$ and $c+d$ even when 
$i,j \geq 2$.  When $\ell$ is even, 
we have changed the domain of integration 
from $\{(x,y) \, | \, 0 \leq x \leq n x_{\ell n}/2, 
y_{\ell n} x / n x_{\ell n} \leq y \leq 
y_{\ell n} x / n x_{\ell n} + y_{\ell n}\}$ to 
$\{(x,y) \, | \, 0 \leq x \leq n x_{\ell n}/2, 0 \leq y \leq 
y_{\ell n}\}$, as this will not 
affect the value of the integral.

We list the following facts without proof, as they follow 
easily from the symmetries  
$V(x,y) = V(-x,y) = V(x,-y) = V((x_{\ell n}/2)-x,y) = 
V(x,(y_{\ell n}/2)-y)$ of $V$.   

\begin{itemize}
\item If $i+j$ is odd, then $b_{ij}=b_{ji}=0$.  
\item If $j$ is odd, and if $a$ and $b$ are not both even, then 
       $b_{1j}=b_{j1}=0$.  
\item For $b_{ij} \neq 0$, 
     \[ b_{ij}=b_{ji} = 
   \sum_{0 \leq A,B \leq |a|+|b|+|c|+|d|, \; A,B \mbox{ even}} 
    c_{i,j,A,B} \cdot {\cal I}(A,B) \; , \] where $c_{i,j,A,B} \in 
   \bfZ$ if $i=j=1$ or $i,j \geq 2$, and $c_{i,j,A,B}/\sqrt{2} 
   \in \bfZ$ if precisely one of $i$ or $j$ is $1$, and 
\[ {\cal I}(A,B) :=  \frac{1}{nx_{\ell n}y_{\ell n}} 
\int_0^{n x_{\ell n}} \int_0^{y_{\ell n}} V(x,y) 
(\cos \frac{2 \pi x}{n x_{\ell n}})^A 
(\cos \frac{2 \pi y}{y_{\ell n}})^B dxdy \; . \] 
The $c_{i,j,A,B}$ are easily determined by repeated use 
of the identities 
$\cos(\alpha \pm \beta) = \cos(\alpha) \cos(\beta) \mp 
\sin(\alpha) \sin(\beta)$, 
$\sin(\alpha \pm \beta) = \cos(\alpha) \sin(\beta) \pm 
\cos(\alpha) \sin(\beta)$, and $\sin^2(\alpha) = 1 - 
\cos^2(\alpha)$.  
\end{itemize}

So, for $\ell$ odd (resp. $\ell$ even), if we compute 
$A_{m=2l^2 - 2l + 1}$ (resp. $A_{m=4l^2 - 4l + 
1}$) for some $l \in \bfZ^+$, then we need compute only 
$l(l+1)/2$ (resp. $2l^2-l$) integrals of the form ${\cal I}(A,B)$.  

\begin{remark}
We can compute the number of negative 
eigenvalues of $A_m$.  By Theorem 6.1 and Lemma 2.4 this 
number or one less than this number estimates 
Ind(${\cal W}_{\ell/n}$).  These estimates are listed in 
the third column of Table 3.  The choice of $A_m$ is listed 
in this table's fourth column.  For example, for 
${\cal W}_{3/2}$, we chose $A_m = A_{181}$, and computed 
numerically that $A_{181}$ has $11$ negative eigenvalues.  
Thus Ind(${\cal W}_{3/2}$) is estimated to be 
either $10$ or $11$, as written 
in the third column.  And those $11$ negative eigenvalues 
range from $\lambda_1^{(181)} \approx -35.4$ to 
$\lambda_{11}^{(181)} \approx -1.47$ when $H = 
1/2$, as written in the fifth column.  (Changing 
$H$ changes the eigenvalues by a factor of $H$, so the 
number of negative eigenvalues is independent of $H$.)  

The size of 
$A_m$ varies with the choice of $\ell/n$, since 
we choose $A_m$ as large as possible in each 
case without creating significant numerical errors.  
(Note that as the size of $A_m$ increases, the 
$c_{i,j,A,B}$ quickly become very large, hence even very 
small numerical errors in the estimates of ${\cal I}(A,B)$ can 
result in large numerical errors, for large $m$.)  

By Lemma 2.5, Null(${\cal W}_{\ell/n}$)$\geq 6$, 
hence the first six nonnegative 
eigenvalues of $A_m$ will converge to zero as $m \to \infty$.  
This provides some indication of how close the 
eigenvalues of $A_m$ are to the first $m$ eigenvalues of ${\cal L}$.  
Hence, in Table 3's final column, we include the 
range of the first six 
positive eigenvalues of $A_m$.  For example, the first 
six positive eigenvalues of the matrix $A_{181}$ for 
${\cal W}_{3/2}$ range from 
$\lambda_{12}^{(181)} \approx 0.059$ to 
$\lambda_{17}^{(181)} \approx 0.65$ when $H=1/2$.  
\end{remark}

\end{document}